\documentclass[10pt]{amsart}

\usepackage[all]{xy}
\usepackage{amsmath, amssymb, hyperref}
\usepackage{tikz}

\newtheorem{thm}{Theorem}[section]
\newtheorem{prop}[thm]{Proposition}

\newtheorem{cor}[thm]{Corollary}

\theoremstyle{remark}
\newtheorem{remark}[thm]{Remark}

\theoremstyle{definition}
\newtheorem{definition}[thm]{Definition}

\makeatletter
\renewcommand*\env@matrix[1][*\c@MaxMatrixCols c]{%
  \hskip -\arraycolsep
  \let\@ifnextchar\new@ifnextchar
  \array{#1}}
\newcommand*\isom{\xrightarrow{\sim}}
\newcommand{\pair}[1]{\langle#1\rangle}

\newcommand{\divisor}{\operatorname{div}}

\newcommand{\Div}{\operatorname{Div}}
\newcommand{\codim}{\operatorname{codim}}

\newcommand{\Spec}{\operatorname{Spec}}

\def\qq{\mathbb{Q}}

\def\rr{\mathbb{R}}
\def\zz{\mathbb{Z}}
\def\cc{\mathbb{C}}

\def\nn{\mathbb{N}}
\def\mm{\mathcal{M}}
\def\ll{\mathcal{L}}
\def\bb{\mathcal{B}}

\def\oo{\mathcal{O}}

\def\xx{\tilde{X}}
\def\yy{\tilde{Y}}

\def\CH{CH}

\def\Kbar{\bar{K}}
\def\qbar{\bar{\qq}}

\def\Mhat{\hat{\mm}}

\def\can{\mathrm{can}}
\def\vareps{\varepsilon}

\def\h{\mathrm{h}}
\def\pp{\mathcal{P}}

\numberwithin{equation}{section}

\begin{document}

\title[N\'eron-Tate heights of cycles on jacobians]{N\'eron-Tate heights of cycles on jacobians}

\author{Robin de Jong}

\begin{abstract} We develop a method to calculate the N\'eron-Tate height of tautological integral cycles on jacobians of curves defined over number fields. As examples we obtain closed expressions for the N\'eron-Tate height of the difference surface, the Abel-Jacobi images of the square of the curve, and of any symmetric theta divisor. As applications we obtain a new effective positive lower bound for the essential minimum of any Abel-Jacobi image of the curve and a proof, in the case of jacobians, of a formula proposed by Autissier relating the Faltings height of a principally polarized abelian variety with the N\'eron-Tate height of a symmetric theta divisor.   
\end{abstract}



\maketitle

\thispagestyle{empty}

\section{Introduction}

Our main aim in this paper is to develop a method to calculate the N\'eron-Tate heights $\h'_\ll(Z_{m,\alpha})$ of certain tautological integral cycles $Z_{m,\alpha}$ on jacobians $J$ of curves $X$ defined over a number field. The cycles $Z_{m,\alpha}$ are obtained as images inside $J$ of small cartesian powers $X^r$ of the curve. Here $m$ is an $r$-tuple of non-zero integers, and $\alpha$ is a chosen base divisor of degree one on $X$. We obtain closed expressions for the height of the difference surface inside $J$, the Abel-Jacobi images of the square of the curve, and of any symmetric theta divisor. 

In many cases, including the case of the Abel-Jacobi images of the curve $X$ itself, we are able to derive a positive lower bound for the N\'eron-Tate height of $Z_{m,\alpha}$. We recall that such lower bounds yield effective versions of the (generalized) Bogomolov conjecture for $Z_{m,\alpha}$. The closed expression that we obtain for the N\'eron-Tate height of a symmetric theta divisor leads to a proof, in the case of jacobians, of a formula, proposed around ten years ago by P. Autissier \cite{aut}, relating the Faltings height of a principally polarized abelian variety with the N\'eron-Tate height of a symmetric theta divisor.

Let $A$ be an abelian variety defined over $\qbar$, and let $\ll$ be a symmetric ample line bundle on $A$ defining a principal polarization on $A$. Then to each closed subvariety $Z$ of $A$ one has associated its N\'eron-Tate height $\h'_\ll(Z) \in \rr$, measuring in an intrinsic way the arithmetic complexity of $Z$. The case of points $Z$ being classical, approaches to define $\h'_\ll(Z)$ for higher dimensional $Z$ were only developed much more recently in work of P. Philippon \cite{phil}, W. Gubler \cite{gu}, J.-B. Bost, H. Gillet and C. Soul\'e \cite{bostintrinsic} \cite{bgs}, and S. Zhang \cite{zhsmall}. Two important properties of the N\'eron-Tate height are: for all integers $n>1$ and all subvarieties $Z\subset A$ we have a quadratic relation $\h'_\ll([n](Z))=n^2\,\h'_\ll(Z)$, and moreover we always have $\h'_\ll(Z) \geq 0$. 

One well known application of the N\'eron-Tate height has been in the study of the (generalized) Bogomolov conjecture for subvarieties of $A$. Let $Z \subset A$ be a closed subvariety and let
\[ e'_\ll(Z) = \sup_{Y \subset Z, \codim Y=1} \inf_{x \in (Z\setminus Y)(\qbar)} \h'_\ll(x)  \]
be the so-called (first) essential minimum of $Z$. Then Zhang showed in \cite{zhsmall} the fundamental inequalities
\begin{equation} \label{ineqs} 
(\dim Z +1) \,\h'_\ll(Z) \geq e'_\ll(Z) \geq \h'_\ll(Z) 
\end{equation}
between $e'_\ll(Z)$ and the N\'eron-Tate height $\h'_\ll(Z)$ of $Z$. In particular we have the equivalence 
\begin{equation} \label{equiv} e'_\ll(Z) > 0 \Longleftrightarrow \h'_\ll(Z) > 0 \, .  
\end{equation}
The generalized Bogomolov conjecture, proved by E. Ullmo and Zhang in \cite{ul} \cite{zhann} using equidistribution techniques, states that both conditions in (\ref{equiv}) are satisfied if and only if $Z$ is not a translate by a torsion point of an abelian subvariety of $A$. An effective positive lower bound for $\h'_\ll(Z)$ or $e'_\ll(Z)$ is usually called an ``effective Bogomolov-type result'' for $Z$. The methods of \cite{ul} \cite{zhann} do not provide such lower bounds, but later on such lower bounds have been given for general $Z$, starting with a fundamental work of S. David and P. Philippon \cite{cvv} \cite{dpI} \cite{dpII} \cite{ga}. In a rough form, these lower bounds are of the shape $\h'_\ll(Z) > B (\deg_\ll Z)^{-C}$ where $B, C$ are positive constants depending only on $A$ and $\ll$.

In this paper we are interested in the particular case of jacobians. Thus, let $X$ be a smooth projective connected curve of genus $g\ge 1$ defined over~$\bar{\qq}$, and let $(J,\lambda)$ denote the jacobian of $X$, endowed with its canonical principal polarization. Choose a divisor $\alpha$ of degree one on $X$. We consider integral cycles $Z_{m,\alpha}$ on $J$ obtained as the images of maps
\[ f =f_{m,\alpha} \colon X^r \longrightarrow J \, , \quad (x_1,\ldots,x_r) \mapsto 
\left[\sum_{i=1}^r m_i x_i - d \alpha\right] \, , \]
where $0 \leq r \leq g$ and where $m=(m_1,\ldots,m_r)$ is an $r$-tuple of non-zero integers with $d = \sum_{i=1}^r m_i$. We note that the maps $f$ are always generically finite. If $d=0$ we usually just write $f_m$ and $Z_m$, and for all $r \in \zz_{>0}$ we simply write $f_{r,\alpha}$ and $Z_{r,\alpha}$ in the fundamental case where $m$ is the all-$1$ tuple of length $r$.

Let $\ll$ be a symmetric ample line bundle on $J$ defining $\lambda$. Let $k \subset \qbar$ be a number field such that $X$ is defined and has semistable reduction over $k$. We note that such a number field exists by the semistable reduction theorem of Deligne-Mumford. Our aim is to express the N\'eron-Tate height $\h'_\ll(Z_{m,\alpha})$ of $Z_{m,\alpha}$ in terms of admissible arithmetic intersection theory on models of $X$ over the ring of integers of $k$ in the sense of Zhang \cite{zhadm} \cite{zhsmall}. For the cycle $Z_{1,\alpha}$ such an expression is already well known. 

Namely, let $\hat{\omega}$ be the relative dualizing sheaf of $X$ equipped with its canonical admissible metric \cite[Section~4]{zhadm}. Assume $g \geq 2$ and put $x_\alpha=\alpha - \frac{1}{2g-2}K_X$, where $K_X$ is a canonical divisor on $X$. Then the identity 
\begin{equation} \label{heightcurve} \h'_\ll(Z_{1,\alpha}) =  \frac{1}{8(g-1)[k:\qq]} 
\pair{\hat{\omega},\hat{\omega}} + \frac{g-1}{g}\h'_\ll(x_\alpha)  
\end{equation}
holds in $\rr$, as is proved in for instance \cite[Theorem 5.6]{zhadm} or \cite[Theorem~3.9]{zhsmall}. Here $\pair{\hat{\omega},\hat{\omega}}$ denotes the self-intersection of the relative dualizing sheaf on $X$ in the admissible sense \cite[Section 5.4]{zhadm}. We can further express the N\'eron-Tate height $\h'_\ll(x_\alpha)$ of the degree zero divisor $x_\alpha$ in terms of admissible intersection theory on $X$ by the celebrated Faltings-Hriljac formula, or Hodge Index theorem \cite[Section~5.4]{zhadm}.

Our approach is inspired by Zhang's work \cite{zhgs} on the height of Gross-Schoen cycles. We start by briefly reviewing this work. Let $M(k)_0$ denote the set of finite places of $k$, let $M(k)_\infty$ denote the set of complex embeddings of $k$, and let $M(k)= M(k)_0 \cup M(k)_\infty$.  An important role in \cite{zhgs} is played by a new real-valued local invariant $\varphi(X_v)$ of $X$ for all $v \in M(k)$. We refer to Section \ref{phi_explained} below for a precise definition of $\varphi(X_v)$. For now, we mention that we have $\varphi(X_v)=0$ for all $v \in M(k)$ if $g=1$, and that we have $\varphi(X_v)=0$ if $v$ is non-archimedean and $X$ has good reduction at $v$. For $v \in M(k)_\infty$ the invariant $\varphi(X_v)$ has independently been introduced and studied by N. Kawazumi in the papers \cite{kawhandbook} \cite{kaw}. For $v \in M(k)_0$ the invariant $\varphi(X_v)$ can be calculated from the polarized dual graph of the reduction of $X$ at $v$.

For $v \in M(k)_0$ we put $Nv= \#\kappa(v)$, where $\kappa(v)$ is the residue field at $v$, and for $v \in M(k)_\infty$ we put $Nv=\mathrm{e}$. We then have a well-defined global invariant $\sum_{v \in M(k)} \varphi(X_v) \log Nv$ of $X$. 
This global invariant shows up in \cite{zhgs} in a formula for the height $\pair{\Delta_\alpha,\Delta_\alpha}$ in the sense of Beilinson-Bloch \cite{beil} \cite{bljpaa} of the Gross-Schoen cycle $\Delta_\alpha \in CH(X^3)_0$ associated to $X$ and $\alpha$ (see \cite{gs} or \cite {zhgs} for a definition). Zhang's formula (cf. \cite[Theorem 1.3.1]{zhgs}) reads\footnote{We note that in \cite{zhgs} there appears a coefficient $6(g-1)$ in front of $[k:\qq]\,\h'_\ll(x_\alpha)$. The difference is due to a different normalization of the N\'eron-Tate height of a point.}
\begin{equation} \label{htgs} \pair{\Delta_\alpha,\Delta_\alpha} = \frac{2g+1}{2g-2} \pair{\hat{\omega},\hat{\omega}} - \sum_{v \in M(k)} \varphi(X_v) \log Nv + 12(g-1)\,[k:\qq]\,\h'_\ll(x_\alpha)  \, . 
\end{equation}
The arithmetic analogue of Grothendieck's Standard Conjecture of Hodge Index type proposed by Gillet-Soul\'e \cite{gsstandard}  predicts that $\pair{\Delta_\alpha,\Delta_\alpha} \geq 0$, which is however not known to be true in general. By (\ref{htgs}) the arithmetic Standard Conjecture implies the lower bound
\begin{equation} \label{conjlowerbound} \pair{\hat{\omega},\hat{\omega}} \geq \frac{2g-2}{2g+1} \sum_{v \in M(k)} \varphi(X_v) \log Nv 
\end{equation}
for $\pair{\hat{\omega},\hat{\omega}}$. As it turns out, the right hand side of this inequality is strictly positive if $g \geq 2$. Indeed, if $g \geq 2$ we have $\varphi(X_v)>0$ for each $v \in M(k)_\infty$, and moreover by a deep result due to Z. Cinkir \cite{ciinv} we have if $g\geq 2$ for each $v \in M(k)_0$ the inequality
\[ \varphi(X_v)\geq  c(g) \delta_0(X_v) + \sum_{i=1}^{[g/2]} \frac{2i(g-i)}{g}\delta_i(X_v)  \, , \]
where $c(g)$ is some positive constant depending only on $g$ and where, following traditional notation, we denote by $\delta_0(X_v)$ the number of non-separating geometric double points on the reduction of $X$ at $v$, and by $\delta_i(X_v)$ for $i=1,\ldots,[g/2]$ the number of geometric double points on the reduction of $X$ at $v$ whose local normalization has two connected components, one of arithmetic genus $i$, and one of arithmetic genus $g-i$. 

It follows that the conjectural inequality in (\ref{conjlowerbound}) states an effective positive lower bound for $\pair{\hat{\omega},\hat{\omega}}$, and therefore, by equation (\ref{heightcurve}), an effective Bogomolov-type result for $Z_{1,\alpha}$. We will show in 
Corollary \ref{effectivebog} below that the lower bound in (\ref{conjlowerbound}) actually holds with the coefficient $(2g-2)/(2g+1)$ replaced by $2/(3g-1)$. Hence, we unconditionally have an effective Bogomolov-type result for $Z_{1,\alpha}$.

If $X$ is hyperelliptic and $x_\alpha=0$, the Gross-Schoen cycle $\Delta_\alpha$ is rationally equivalent to zero \cite[Proposition~4.8]{gs}, and hence has vanishing Beilinson-Bloch height. We conclude that for hyperelliptic curves $X$ the equality
\[ \pair{\hat{\omega},\hat{\omega}} = \frac{2g-2}{2g+1} \sum_{v \in M(k)} \varphi(X_v) \log Nv \]
holds.

In this paper we will elaborate on Zhang's method in \cite{zhgs} that leads to his formula (\ref{htgs}) for the height of the Gross-Schoen cycle.
We will show that it allows more generally to compute the height of any of the cycles $Z_{m,\alpha}$ introduced above. More precisely, we will show the following result.
\begin{thm} \label{mainintro}
Assume $g \geq 2$. Let $0 \leq r \leq g$ and let $m=(m_1,\ldots,m_r)$ be an $r$-tuple of non-zero integers. There exist rational numbers $a, b, c$ depending only on $m$ and $g$ such that for all number fields $k$, all smooth projective geometrically connected curves $X$ of genus $g$ with semistable reduction over $k$, and all $\alpha \in \Div^1 X$ the equality
\[ \h'_\ll(Z_{m,\alpha}) = \frac{1}{[k:\qq]} \left( a \, \pair{\hat{\omega},\hat{\omega}} + b \, \sum_{v \in M(k)} \varphi(X_v) \log N v \right) + c \, \h'_\ll(x_\alpha) \]
holds. Here $x_\alpha=\alpha - \frac{1}{2g-2}K_X$.
\end{thm}
\begin{cor} Let $X$ be a smooth projective connected curve over $\qbar$ of genus $g \geq 2$, and let $\alpha \in \Div^1 X$. Let $V$ be the $\qq$-span inside $\rr$ of all 
$\h'_\ll(Z_{m,\alpha})$ where $m$ is an $r$-tuple of non-zero integers with $r \leq g$. Then $\dim_\qq V \leq 3$. If $X$ is hyperelliptic, we have $\dim_\qq V \leq 2$.
\end{cor}

Apart from \cite{zhgs}, our proof of Theorem \ref{mainintro} is heavily inspired by a method developed recently by R.~Wilms in the context of the moduli space of curves \cite{wi}. In particular, our proof provides an algorithm to compute $a$, $b$ and $c$ from $m$ and~$g$. Upon hearing about our proof, D.~Holmes implemented the algorithm in \verb+SAGE+, and kindly made it publicly available via \verb+https://arxiv.org/abs/1610.01932+. From running the algorithm it seems not straightforward in general to give a closed expression for $a$, $b$ or $c$ in terms of $m$ and~$g$. We will illustrate Theorem \ref{mainintro} by exhibiting some particular examples, and deducing some consequences, in the cases $r=2$ and $r=g-1$, the case $r=1$ being essentially dealt with by (\ref{heightcurve}).

Here is first of all what we find if we set $r=2$. 
\begin{thm} \label{squareintro} Let $X$ be a smooth projective geometrically connected curve of genus $g \geq 2$ with semistable reduction over the number field $k$.  Let $F=Z_{(1,-1)}$ be the image of $X^2$ in $J$ under the map $X^2 \to J$ given by $(x,y) \mapsto [x-y]$. Then the equality
\[ 12g(g-1)\, [k:\qq] \, \h'_\ll(F) = (3g-1)\,\pair{\hat{\omega},\hat{\omega}} - 2 \sum_{v \in M(k)} \varphi(X_v) \log Nv \]
holds. Next, let $\alpha \in \mathrm{Div}^1 X$, and let $Z_{2,\alpha}$ be the image of $X^2$ in $J$ under the map $X^2 \to J$ given by $(x,y) \mapsto [x+y-2\alpha]$.   Then the equality
\[ \begin{split} 12g(g-1)& \,[k:\qq] \, \h'_\ll(Z_{2,\alpha})  = \frac{3g^2-8g-1}{g-1}\pair{\hat{\omega},\hat{\omega}} \\ & + 2 \sum_{v \in M(k)} \varphi(X_v)\log Nv + 48(g-1)(g-2) \,[k:\qq]\,\h'_\ll(x_\alpha) \end{split} \]
holds.
\end{thm}
As $\h'_\ll(F) \geq 0$ we immediately deduce from the first equation the following result.
\begin{cor} \label{effectivebog} Let $X$ be a smooth projective geometrically connected curve of genus $g \geq 2$ with semistable reduction over the number field $k$. Then the inequalities
\[ \pair{\hat{\omega},\hat{\omega}} \geq \frac{2}{3g-1} \sum_{v \in M(k)} \varphi(X_v)\log Nv > 0 \]
hold.
\end{cor}
By equation (\ref{heightcurve}) we obtain from Corollary \ref{effectivebog} an effective Bogomolov-type result for $Z_{1,\alpha}$. More generally, whenever one has $a \geq 0$, $b \geq 0$ and $a+b >0$ in the identity of Theorem \ref{mainintro}, one finds by Corollary \ref{effectivebog} an effective Bogomolov-type result for $Z_{m,\alpha}$. Experimenting with the \verb+SAGE+ implementation of our algorithm has so far not led to an  improvement of the lower bound in Corollary~\ref{effectivebog}, and we believe that the lower bound in Corollary \ref{effectivebog} is in fact the best possible one that one can obtain by the methods described in this paper.

Upper bounds for $\h'_\ll(F)$ and $\h'_\ll(Z_{2,\alpha})$ play an important role in a recent work \cite{pa} by P. Parent. 
For example, a combination of Theorem \ref{squareintro} and Corollary \ref{effectivebog} yields the estimate
\[ \h'_\ll(Z_{2,\alpha}) \leq \frac{1}{[k : \qq]} \frac{g-2}{2(g-1)^2} \pair{\hat{\omega},\hat{\omega}} + \frac{4(g-2)}{g} \h'_\ll(x_\alpha) \, , \]
and then using equation (\ref{heightcurve}) we find
\[  \h'_\ll(Z_{2,\alpha}) \leq \frac{4(g-2)}{g-1} \, \h'_\ll(Z_{1,\alpha}) \, . \]
This estimate may be used to simplify some of the arguments in \cite[Section~5.1]{pa}. 

The next example that we consider is the case of a symmetric theta divisor inside~$J$. We note that such theta divisors can be identified with cycles $Z_{g-1,\alpha} \subset J$ where $\alpha \in \mathrm{Div}^1 X$ is such that $x_\alpha=0$, in particular we are dealing with the case that $r=g-1$. We find the following result.
\begin{thm} \label{thetaheight}
Let $X$ be a smooth projective geometrically connected curve of genus $g\geq 1$ with semistable reduction over a number field $k$, and denote by $(J,\lambda)$ its jacobian. Let $\Theta$ be an effective symmetric ample divisor on $J$ that defines the principal polarization $\lambda$, and put $\ll=\oo_J(\Theta)$. Then the identity
\[ 2g \,[k:\qq] \, \h'_\ll(\Theta) = \frac{1}{12} \pair{\hat{\omega},\hat{\omega}} + \frac{1}{6} \sum_{v \in M(k)} \varphi(X_v) \log Nv \]
holds.
\end{thm}
More generally, let $(A,\lambda)$ be a principally polarized abelian variety over $\qbar$ of dimension $g \geq 1$. Then one has naturally associated to $(A,\lambda)$ the N\'eron-Tate height $\h'_\ll(\Theta)$ of any symmetric effective divisor $\Theta$ defining $\lambda$ on $A$, where $\ll=\oo_A(\Theta)$. Another natural invariant associated to $(A,\lambda)$ is the stable Faltings height $\h_F(A)$ as introduced in \cite{famordell}. It is natural to ask in general how $\h'_\ll(\Theta)$ and $\h_F(A)$ are related. This question has been investigated by P. Autissier in his paper \cite{aut}. Before we can state Autissier's main result, we need to introduce some more notation. 

Assume that $A$ and $\ll$ are defined over a number field $k \subset \qbar$, and assume that an admissible adelic metric $(\|\cdot\|_v)_{v \in M(k)}$ has been chosen on $\ll$ (see Section \ref{admissible} for definitions). Let $s$ be any non-zero global section of $\ll$ on $A$. For each $v \in M(k)_\infty$ we then put
\[ I(A_v,\lambda_v) = -\int_{A_v(\cc)} \log \|s\|_v \, \mu_v + \frac{1}{2} \log \int_{A_v(\cc)} \|s\|_v^2 \, \mu_v \, ,\]
where $\mu_v$ denotes the Haar measure on the complex torus $A_v(\cc)$. The real-valued local invariant $I(A_v,\lambda_v)$ is independent of the choice of $s$ and $\ll$, and an application of Jensen's inequality yields the strict inequality $ I(A_v,\lambda_v) > 0 $.

The main result of \cite{aut} is then the following. As above let $\h_F(A)$ denote the stable Faltings height of $A$, and put $\kappa_0 = \log(\pi \sqrt{2})$. Assume that $A$ has good reduction over the number field $k$. Then one has the identity
\begin{equation} \label{aut} \h_F(A) = 2g \, \h'_\ll(\Theta) - \kappa_0 g + \frac{2}{[k:\qq]} \sum_{v \in M(k)_\infty} I(A_v,\lambda_v) \log Nv  
\end{equation}
in $\rr$. 

Autissier's result is based on the ``key formula'' for principally polarized abelian schemes. Naturally, one wants to remove the condition that $A$ has (potentially) everywhere good reduction. In fact, Autissier proposes in \cite[Question]{aut} that for  
\emph{any} principally polarized abelian variety $(A,\lambda)$ of dimension $g \geq 1$ with semistable reduction over some number field $k$ one should have a natural identity of the type
\begin{equation} \label{autissconj} \begin{split} \h_F(A) = &\,\, 2g \, \h'_\ll(\Theta) +\frac{1}{[k:\qq]} \sum_{v \in M(k)_0} \alpha_v \log Nv \\ & - \kappa_0 g  + \frac{2}{[k:\qq]} \sum_{v \in M(k)_\infty} I(A_v,\lambda_v) \log Nv 
\end{split} \end{equation}
where $\alpha_v$ is a non-negative rational number that can be calculated from the reduction of $A$ at $v$, with $\alpha_v=0$ if $A$ has good reduction at $v$. 

In \cite{aut} Autissier proved his proposed formula (\ref{autissconj}) for all principally polarized abelian varieties of dimension one or two, and for  arbitrary products of such. In these cases he was also able to give an explicit description for the local non-archimedean factors~$\alpha_v$. Our Theorem \ref{autforjac} below will give identity (\ref{autissconj}) for (products of) arbitrary jacobians, with an explicit description of the $\alpha_v$ in terms of the reduction graph of the underlying curves at $v$. 

Let $v \in M(k)$ be a finite place. We recall \cite{zhadm} that the dual graph of the reduction of $X$ at $v$ is naturally a metrized graph. Let $\Gamma$ be any metrized graph. Then viewing $\Gamma$ as an electric circuit we have for each pair $x$, $y$ of points of $\Gamma$ the effective resistance $r(x,y)$ between~$x$ and~$y$. Taking $x$ as a base point the function $r(x,y)$ is piecewise quadratic in $y \in \Gamma$ and has a well-defined Laplacian $\Delta_y \,r(x,y)$ in the sense of \cite[Appendix]{zhadm}. We put
\begin{equation} \label{crmeasure}   \mu_{\mathrm{can}}(y)  = \frac{1}{2} \,\Delta_y \,r(x,y) + \delta_x(y) \, . 
\end{equation}
Then $\mu_{\mathrm{can}}$ is a signed measure on $\Gamma$, independent of the choice of $x \in \Gamma$. The canonical measure $\mu_{\mathrm{can}}$ was first introduced and studied by T.~Chinburg and R.~Rumely \cite{cr} who also give an explicit formula for it. We write $\tau(\Gamma)$ for the natural real-valued invariant 
\begin{equation} \label{tauinv} \tau(\Gamma) =  \frac{1}{2} \int_\Gamma r(x,y) \, \mu_{\mathrm{can}}(y)  
\end{equation}
of $\Gamma$, which can also be shown \cite{cr} to be independent of the choice of $x \in \Gamma$. The invariant $\tau(\Gamma)$ and its applications in arithmetic geometry have been studied in various papers, including \cite{bffourier} \cite{br} \cite{cr} \cite{ciadm} \cite{ciop}.

We note that $\tau(\Gamma)$ is rational if in some model of $\Gamma$ all edges have rational lengths. When $X$ is a smooth projective geometrically connected curve with semistable reduction over a number field $k$ and $v \in M(k)$ is a finite place, we simply write $\tau(X_v)$ for the tau-invariant of the dual graph of the reduction of $X$ at $v$. In a similar vein we write $\delta(X_v) = \sum_{i=0}^{[g/2]} \delta_i(X_v)$ for the total length of the dual graph at $v$, or equivalently the number of geometric singular points on the reduction of $X$ at $v$. Both $\tau(X_v)$ and $\delta(X_v)$ are non-negative rational numbers, and vanish if $X$ has good reduction at $v$.

With these definitions we have the following result, answering Autissier's question in the affirmative for jacobians.
\begin{thm} \label{autforjac} Let $X$ be a smooth projective geometrically connected curve of genus $g \geq 1$ with semistable reduction over a number field $k$. Let $(J,\lambda)$ be the jacobian of $X$. 
For each $v \in M(k)_0$ we define the rational number 
\[ \alpha_v = \frac{1}{8}\delta(X_v) - \frac{1}{2} \tau(X_v) \, .  \]
Then for each $v \in M(k)_0$ we have $\alpha_v \geq 0$. Moreover we have $\alpha_v=0$ if and only if the jacobian $J$ has good reduction at $v$. Finally the identity
\[ \begin{split} \h_F(J) = &\,\, 2g \, \h'_\ll(\Theta) +\frac{1}{[k:\qq]} \sum_{v \in M(k)_0} \alpha_v \log Nv \\ & - \kappa_0 g + \frac{2}{[k:\qq]} \sum_{v \in M(k)_\infty} I(J_v,\lambda_v) \log Nv \end{split} \]
holds in $\rr$.
\end{thm}
We note that the inequality $\delta(\Gamma) \geq 4\tau(\Gamma)$ for metrized graphs $\Gamma$ is due to Rumely and has been known for some time \cite[Equation (14.3)]{br} \cite[Equation (10)]{ciop}. 

Our proofs of Theorems \ref{thetaheight} and \ref{autforjac} rely on Autissier's result (\ref{aut}) in the potentially everywhere good reduction case, the Noether formula for arithmetic surfaces, and the remarkable identity 
\[ \delta_F(X_v) - 4g\log(2\pi)= -12\,\kappa_0 g + 24\,I(J_v,\lambda_v)+ 2\,\varphi(X_v)  \]
for $v \in M(k)_\infty$ found recently by R. Wilms in the already quoted paper \cite{wi}, expressing the Faltings delta-invariant $\delta_F(X_v)$ from \cite[p.~401]{fa} in terms of the invariants $I(J_v,\lambda_v)$ and $\varphi(X_v)$. \\

The plan of this paper is as follows. In Section \ref{prelims} we review the necessary notions and results from intersection theory of integrable adelic line bundles from Zhang's paper \cite{zhsmall}. We will assume that the reader is familiar with arithmetic intersection theory in the style of H. Gillet and C. Soul\'e \cite{bgs} \cite{giso}. In Section \ref{sec:projection} we show how several projection formulae from \cite{bgs} \cite{giso} generalize to the context of integrable line bundles. These results seem to be of independent interest. We include a study of the Deligne pairing of integrable line bundles along flat projective morphisms.

In Section \ref{admissible} we consider a special class of integrable line bundles on curves and polarized abelian varieties, namely the admissible bundles. On curves, the admissible line bundles are precisely those constructed in \cite{zhadm} by considering Green's functions on the reduction graph of a curve. We will see how the N\'eron-Tate height of a cycle on a polarized abelian variety can be expressed in terms of admissible intersection theory. In Section \ref{phi_explained} we review from \cite{zhgs} the $\varphi$-invariant and some of its properties. 

From Section \ref{NTheightspecial} we start with actual computations. We first derive in Theorem~\ref{main} a general formula for the N\'eron-Tate height of a tautological integral cycle in terms of Deligne pairings of suitable integrable line bundles. In Section \ref{proofmain} we then deduce Theorem \ref{mainintro} from Theorem \ref{main}. We focus on the particular case of two-dimensional cycles in Section \ref{surfaces}, leading to a proof of Theorem \ref{squareintro}. Finally in Section \ref{proofautforjac} we prove Theorems \ref{thetaheight} and \ref{autforjac} about the symmetric theta divisor.

In this paper, a variety is an integral scheme of finite type over a field, and a curve is a variety of dimension one.\\

\noindent {\sc Acknowledgements}---I would like to thank David Holmes for valuable remarks and discussions, and for implementing the method described in this paper in \verb+SAGE+. 

\section{Integrable line bundles} \label{prelims}

We recall the notion of integrable adelic line bundle from \cite{zhsmall} and discuss the arithmetic intersection number of a tuple of integrable line bundles. In particular this leads to the notion of height of a closed subvariety with respect to an ample integrable line bundle, cf. Definition \ref{height}. Although we work mainly from \cite{zhsmall} we also recommend \cite{cl} \cite{clmeas} for a different perspective on the material of this section (and the next ones), making a more consistent use of Berkovich analytic spaces to deal with the metrics at the finite places. 
We assume throughout that the reader is familiar with the basics of arithmetic intersection theory in the style of Gillet-Soul\'e \cite{bgs} \cite{giso}.  A word on terminology: what is here and in \cite{zhsmall} called an algebraic line bundle corresponds to what is called a smooth line bundle in \cite{cl}, and what is here and in \cite{zhsmall} called an integrable line bundle corresponds to what is called an admissible line bundle in \cite{cl}.

Let $(K,|\cdot|)$ be a local field, and let $\Kbar$ be an algebraic closure of $K$. Let $X$ be a projective variety over $\Kbar$, and let $\ll$ be a line bundle on $X$. A continuous metric $\| \cdot \|$ on $\ll$ is the datum of a continuously varying family of $\Kbar$-norms $\|\cdot\|_x$ on the fibers $x^*\ll$ of $\ll$, where $x$ moves through the topological space $X(\Kbar)$. If $K$ is archimedean, we always assume that $(\ll,\|\cdot\|)$ is the pullback, along a closed immersion $X \to Y$ of $X$ into a smooth projective variety $Y$, of a smooth hermitian line bundle on the complex manifold $Y(\Kbar)$. In particular, we have a curvature current $c_1(\ll,\|\cdot\|)$ on $X(\Kbar)$. 

If $K$ is non-archimedean, an important class of continuously metrized line bundles on $X$ can be defined via models of $X$ over the ring of integers $R$ of $\Kbar$ as follows. In general, when $\tilde{L}$ is a locally free $R$-module of rank one, we have a natural $\Kbar$-norm $\| \cdot \|_{\tilde{L}}$ on the $\Kbar$-vector space $L = \tilde{L} \otimes_R \Kbar$ by putting $ \| \ell \|_{\tilde{L}} = \inf_{a \in \Kbar} \{ |a| : \ell \in a \tilde{L} \}  $
for all $\ell \in L$. Now let $S = \Spec R$, and let $\pi \colon \tilde{X} \to S$ be a projective flat morphism with generic fiber $X$. Let $\ll$ be a line bundle on $X$, and let $\tilde{\ll}$ be a line bundle on $\tilde{X}$ whose restriction to $X$ is equal to $\ll^{\otimes e}$, for some $e \in \zz_{>0}$. We call $(\tilde{X},\tilde{\ll})$ a model for $(X,\ll^{\otimes e})$. For $x \in X(\Kbar)$ let $\tilde{x} \in \tilde{X}(S)$ be the unique section of $\pi$ that extends $x$ (recall that $\pi$ is projective). Then $\tilde{x}^*\tilde{\ll}$ is a locally free $R$-module of rank one satisfying $\tilde{x}^*\tilde{\ll} \otimes_R \Kbar = x^*\ll^{\otimes e}$ and for $\ell \in x^*\ll$ we then put $\| \ell \|_{\tilde{\ll}} = \| \ell \|_{\tilde{x}^*\tilde{\ll}}^{1/e}$. We say that the continuous metric $\|\cdot\|_{\tilde{\ll}}$ on $\ll$ is algebraic induced by the model $(\tilde{X},\tilde{\ll})$ of $(X,\ll^{\otimes e})$.

Next, let $k$ be a number field. Let $M(k)_0$ denote the set of finite places of $k$, $M(k)_\infty$ the set of complex embeddings of $k$, and write $M(k)= M(k)_0 \cup M(k)_\infty$. For $v \in M(k)$ let $k_v$ denote the completion of $k$ at $v$, in particular we have for each $v \in M(k)$ a canonical embedding $k \to k_v$. For $v \in M(k)_0$ let $p$ be the unique prime of $\zz$ such that  $v$ divides $p$. Then we endow $k_v$ with the unique absolute value $|\cdot|_v$ such that $|p|_v=p^{-[k_v : \qq_p]}$. For $v \in M(k)_\infty$ we put the usual euclidean norm on $k_v$. The collection of absolute values thus obtained for all $v \in M(k)$ satisfies the product formula: for all $a \in k^\times$ we have $\sum_{v \in M(k)} \log |a|_v = 0$.

Let $X$ be a projective variety over $k$, and let $\ll$ be a line bundle on $X$. For each $v \in M(k)$ we denote $X_v=X \otimes \bar{k}_v$ and $\ll_v = \ll \otimes \bar{k}_v$. An adelic metric on $\ll$ is a collection of continuous metrics $\|\cdot\|_v$ on $\ll_v$ on $X_v$ indexed by $v \in M(k)$. Let $e \in \zz_{>0}$. Let $R$ be the ring of integers of $k$. By a model of $(X,\ll^{\otimes e})$ we mean a pair $(\tilde{X},\bar{\ll})$ where $\tilde{X}$ is a projective flat model of $X$ over $R$ and $\bar{\ll}=(\tilde{\ll},(\|\cdot\|_v)_{v \in M(k)_\infty})$ is a smooth hermitian line bundle on $\tilde{X}$ whose underlying line bundle $\tilde{\ll}$ coincides with $\ll^{\otimes e}$ over $X$. A model of $(X,\ll^{\otimes e})$ gives naturally, using for $v \in M(k)_0$ the construction described before, an adelic metric $(\|\cdot\|_v)_{v \in M(k)}$ on $\ll$.

We say that an adelic metrized line bundle $\hat{\ll}=(\ll,(\| \cdot \|_v)_{v \in M(k)})$ on $X$ is approximated by models $(\tilde{X}_i,\bar{\ll}_i)$ of $(X,\ll^{\otimes e_i})$ for $i=1,2,\ldots$ if for all but finitely many $v \in M(k)$ the metric $\|\cdot\|_{i,v}$ induced by $(\tilde{X}_i,\bar{\ll}_i)$ is independent of $i$ and if for all $v\in M(k)$ the functions $\log( \| \cdot \|_{i,v}/\| \cdot \|_v)$ on $X(\bar{k}_v)$ converge uniformly to zero. 
\begin{remark}
Let $(\tilde{X}_i,\bar{\ll}_i)$ be a sequence of models approximating the adelic line bundle $\hat{\ll}$, and assume that for each $i=1,2,\ldots$ we have a birational morphism $\varphi_i \colon \tilde{X}'_i \to \tilde{X}_i$ of models which restricts to an isomorphism over $k$. Then the sequence of models
$(\tilde{X}'_i,\varphi_i^* \bar{\ll}_i)$ also approximates $\hat{\ll}$. Indeed, write $S =\Spec R$ then each section $x_i \colon S \to \tilde{X}_i$ lifts uniquely to a section $x'_i \colon S \to \tilde{X}'_i$ by properness, and then $(x'_i)^* \varphi_i^* \tilde{\ll}_i = x_i^* \tilde{\ll}_i$ so that for each $i=1,2,\ldots$ the algebraic metrics defined by $\tilde{\ll}_i$ and $\tilde{\ll}'_i$ on $\ll$ are the same. In particular, we may assume the models $\tilde{X}_i$ to be normal.
\end{remark}
\begin{remark} Let $(\tilde{X}_{ij},\bar{\ll}_{ij})$ for $j=0,\ldots,r$ be sequences of models approximating the adelic line bundles $\hat{\ll}_0,\ldots,\hat{\ll}_r$. Let $\tilde{X}_i $ be the Zariski closure of the small diagonal of $X \times_k \cdots \times_k X$ inside $ \tilde{X}_{i0} \times_S \cdots \times_S \tilde{X}_{ir}$ for $i=1,2, \ldots$, and denote by $\bar{\ll}'_{ij}$ the restriction to $\tilde{X}_i$ of the pullback of $\bar{\ll}_{ij}$ along the projection onto $\tilde{X}_{ij}$. Then for each $j=0,\ldots,r$ the sequence of models $(\tilde{X}_i,\bar{\ll}'_{ij})$ also approximates $\hat{\ll}_j$.
\end{remark}
Let $\tilde{X}$ be a projective flat scheme over $R$.
We say that a smooth hermitian line bundle $\bar{\ll}=(\tilde{\ll},(\|\cdot\|_v)_{v \in M(k)_\infty})$ on $\tilde{X}$ is semipositive (resp. ample) if $\tilde{\ll}$ is relatively semipositive (resp. relatively ample), and each metric $\|\cdot\|_v$ for $v \in M(k)_\infty$ is the restriction of a smooth hermitian metric with semipositive (resp. positive) curvature form. Let $X$ be a projective geometrically irreducible variety over $k$, then an adelic metrized line bundle $\hat{\ll}=(\ll,(\| \cdot \|_v)_v)$ on $X$ is said to be semipositive (resp. ample) if $\hat{\ll}$ is approximated by models $(\tilde{X}_i,\bar{\ll}_i)$ of $(X,\ll^{\otimes e_i})$ with each $\bar{\ll}_i$ semipositive (resp. ample). Finally we say that the adelic line bundle $\hat{\ll}$ on $X$ is integrable if there exist two semipositive adelic line bundles $\hat{\ll}_1$, $\hat{\ll}_2$ on $X$ such that $\hat{\ll}=\hat{\ll}_1 \otimes \hat{\ll}_2^{\otimes -1}$. 
\begin{remark} Let $f \colon Y \to X$ be a morphism of projective varieties over $k$, and let $\hat{\ll}$ be an integrable line bundle on $X$. Then $f^*\hat{\ll}$ is an integrable line bundle on $Y$. Indeed, for each model $\tilde{X}$ of $X$ there exists a model $\tilde{Y}$ of $Y$ together with a morphism $\tilde{f} \colon \tilde{Y} \to \tilde{X}$ extending $f$. Then if $\bar{\ll}$ is a semipositive smooth hermitian line bundle on $\tilde{X}$, we have that $\tilde{f}^*\bar{\ll}$ is a semipositive smooth hermitian line bundle on $\tilde{Y}$.
\end{remark}
For a projective flat scheme $\tilde{X}$ over $R$ with smooth generic fiber we denote by $\widehat{CH}^i(\tilde{X})$ the arithmetic Chow group of $\tilde{X}$ in degree $i$, cf. \cite[Section~3.3]{giso}. We write $S = \Spec R$. By \cite[Section 2.1]{bgs} we have canonical maps 
\[ \deg \colon \widehat{CH}^0(S) \isom \zz\, , \quad \widehat{\deg} \colon \widehat{CH}^1(S) \to \rr \, . \]
The arithmetic Chow groups $\widehat{CH}^i(S)$ vanish for $i \geq 2$. Let $\tilde{Z} \in \mathrm{Z}_{d+1}(\tilde{X})$ be an integral closed subscheme of dimension $d+1$ on $\tilde{X}$. Then for each collection $\bar{\ll}_0,\ldots,\bar{\ll}_r$ of smooth hermitian line bundles on $\tilde{X}$ we have by \cite[Section 2.3]{bgs} an element 
\[ \langle \bar{\ll}_0,\ldots,\bar{\ll}_r | \tilde{Z} \rangle 
= ( \hat{c}_1(\bar{\ll}_0)\cdots \hat{c}_1(\bar{\ll}_r) | \tilde{Z} ) \]
in $\widehat{CH}^{r+1-d}(S)$. If $d \in \{ r,r+1\}$ we obtain, by taking $\widehat{\deg}$ or $\deg$, an element of $\rr$ or $\zz$, for which we use the same notation. If $\tilde{X}$ is regular, the element $\langle \bar{\ll}_0,\ldots,\bar{\ll}_r | \tilde{Z} \rangle$ in $\widehat{CH}^{r+1-d}(S)$ is constructed in \cite[Section 2.3]{bgs} using arithmetic intersection theory on $\tilde{X}$. If $\tilde{X}$ is not necessarily regular, one way of obtaining $\langle \bar{\ll}_0,\ldots,\bar{\ll}_r | \tilde{Z} \rangle$ is to use the existence of a smooth morphism $\tilde{X}' \to S$, a closed immersion $j \colon \tilde{X} \to \tilde{X}'$ over $S$, and smooth hermitian line bundles $\bar{\ll}'_0,\ldots,\bar{\ll}'_r$ on $\tilde{X}'$ together with isometries $\bar{\ll}_i \isom j^*\bar{\ll}'_i$ for $i=0,\ldots,r$, and to put $\langle \bar{\ll}_0,\ldots,\bar{\ll}_r | \tilde{Z} \rangle =\langle \bar{\ll}'_0,\ldots,\bar{\ll}'_r|j_*(\tilde{Z})\rangle$. It can be shown that the result is independent of the choices of $j$ and $\bar{\ll}'_i$. We refer to \cite[Section 2.3.1, Remark (ii)]{bgs} and \cite[Section 3.2.1, Remark]{bgs} for details. 
The intersection number  $\langle \bar{\ll}_0,\ldots,\bar{\ll}_r | \tilde{Z} \rangle$ is symmetric and multilinear in $\bar{\ll}_0,\ldots,\bar{\ll}_r $.  Moreover, when $d = r+1$ we have
\[ \langle \bar{\ll}_0,\ldots,\bar{\ll}_r | \tilde{Z} \rangle = c_1(\tilde{\ll}_0|_X)\cdots c_1(\tilde{\ll}_r|_X)[Z] \in \zz \, ,  \]
where $Z, X$ denote the generic fibers of $\tilde{Z}, \tilde{X}$.

We can extend these arithmetic intersection numbers to the integrable adelic case as follows. Let $X$ be a smooth projective variety over $k$ and let $\hat{\ll}_0,\ldots,\hat{\ll}_r$ be semipositive adelic line bundles on $X$ approximated by semipositive models $(\tilde{X}_i,\bar{\ll}_{ij})$ of $(X,\ll^{\otimes e_{ij}})$. Let $Z \in \mathrm{Z}_d(X)$ be an integral cycle of dimension $d$ on $X$ and let $\tilde{Z}_i$ be the Zariski closure of $Z$ in $\tilde{X}_i$.  By \cite[Theorem~1.4]{zhsmall} the sequence
$ \langle \bar{\ll}_{i0},\ldots,\bar{\ll}_{ir} | \tilde{Z}_i \rangle/e_{i0}\cdots e_{ir} $
converges, with limit independent of choice of models. We denote the limit by
$ \langle \hat{\ll}_0,\ldots,\hat{\ll}_r | Z \rangle$.
By writing an arbitrary integrable line bundle as a quotient of two semipositive line bundles, the definition of $\langle \hat{\ll}_0,\ldots,\hat{\ll}_r | Z \rangle$ is extended to arbitrary tuples $(\hat{\ll}_0,\ldots,\hat{\ll}_r)$ of integrable line bundles on $X$. When $\hat{\ll}_0=\ldots=\hat{\ll}_r=\hat{\ll}$ we often simply write $\langle \hat{\ll}^{r+1}|Z \rangle$, and 
when $Z=X$ we often write $\langle \hat{\ll}_0,\ldots,\hat{\ll}_r \rangle $. The intersection number  $\langle\hat{\ll}_0,\ldots,\hat{\ll}_r | Z \rangle$ is symmetric and multilinear in $\hat{\ll}_0,\ldots,\hat{\ll}_r $, and vanishes unless $d \in \{ r, r+1 \}$. In case $d = r+1$ we have
\[ \langle \hat{\ll}_0,\ldots,\hat{\ll}_r | Z \rangle = c_1(\ll_{0})\cdots c_1(\ll_{r})[Z] \in \zz \, .  \]
The definition of the intersection number $\langle \hat{\ll}_0,\ldots,\hat{\ll}_r | Z \rangle $ extends in a straightforward manner to arbitrary cycles $Z \in \mathrm{Z}_*(X)$.
\begin{definition} \label{height} Let $Z \subset X$ be a closed subvariety of dimension $r$. Let $\hat{\ll}$ be an integrable ample line bundle on $X$. Then one puts
\[ \h'_{\hat{\ll}}(Z) = 
\frac{1}{[k:\qq]} \frac{ \langle \hat{\ll}^{r+1} | Z \rangle}{\langle \hat{\ll}^r|Z \rangle (r+1) } =
\frac{1}{[k:\qq]} \frac{ \langle \hat{\ll}^{r+1} | Z \rangle}{\deg_{\ll}(Z) (r+1) }  \]
for the (normalized or absolute) height of $Z$ with respect to $\hat{\ll}$.
\end{definition}

\section{Projection formulae} \label{sec:projection}

The purpose of this section is to state and prove some projection formulae for integrable line bundles. Recall that an integrable line bundle is essentially determined by a sequence of models. At each level in the sequence, the usual arithmetic intersection numbers satisfy a number of projection formulae as shown in \cite{bgs} \cite{giso}. These formulae yield analogues for integrable line bundles by passage to the limit. Along the way we include a discussion of the Deligne pairing of integrable bundles.

The first projection formula we state is also observed in \cite[Section~3.1.2]{cl}. 
\begin{prop} \label{projection} 
Let $X, Y$ be smooth projective varieties over the number field $k$. Let $f \colon Y \to X$ be a morphism.  Let $Z \in \mathrm{Z}_*(Y)$ be a cycle on $Y$. Let $\hat{\ll}_0,\ldots,\hat{\ll}_r$ be  integrable adelic line bundles on $X$. Then the adelic line bundles $f^*\hat{\ll}_0,\ldots,f^*\hat{\ll}_r$ are integrable on $Y$, and the equality 
\[ \langle \hat{\ll}_0, \ldots, \hat{\ll}_r | f_*(Z) \rangle = \langle f^*(\hat{\ll}_0),\ldots, f^*(\hat{\ll}_r) | Z \rangle  \]
holds in $\rr$. 
\end{prop}
\begin{proof} Without loss of generality we may assume that the $\hat{\ll}_j$ are semipositive and approximated by semipositive models $(\xx_i,\bar{\ll}_{ij})$ of $(X,\ll_j^{\otimes e_{ij}})$, and that $Z$ is an integral cycle. For each $i=1, 2, \ldots$ there exists a model $\yy_i$ of $Y$ together with a morphism $\tilde{f}_i \colon \yy_i \to \xx_i$ extending the morphism $f$. Then each $f^*\hat{\ll}_j$ is semipositive, approximated by semipositive models $(\yy_i,\tilde{f}_i^*\bar{\ll}_{ij})$. For each $i=1, 2, \ldots$ let $\tilde{Z}_i$ denote the Zariski closure of $Z$ inside $\yy_i$. By the projection formula from \cite[Proposition 3.2.1]{bgs} and the Remark following \cite[Proposition 3.2.1]{bgs} we have for each $i=1, 2, \ldots$ an equality of arithmetic intersection numbers
\[  \langle \bar{\ll}_{i0},\ldots, \bar{\ll}_{ir} | \tilde{f}_{i*}(\tilde{Z}_i) \rangle = \langle \tilde{f}_i^*(\bar{\ll}_{i0}), \ldots, \tilde{f}_i^*(\bar{\ll}_{ir}) | \tilde{Z}_i \rangle    \]
in $\rr$. Dividing both sides by $e_{i0}\cdots e_{ir}$ and taking the limit as $i \to \infty$ we obtain the equality from the proposition. 
\end{proof}
\begin{prop} \label{externalproducts} Let $X, Y$ be smooth projective varieties over the number field $k$. Let $p \colon X \times Y \to X$ and $q \colon X \times Y \to Y$ denote the canonical projections. Let $\hat{\ll}_0,\ldots,\hat{\ll}_r$ be integrable adelic line bundles on $X$, and let $\hat{\mm}_0,\ldots, \hat{\mm}_s$ be integrable adelic line bundles on $Y$. Then the adelic line bundles $p^*\hat{\ll}_0, \ldots, p^*\hat{\ll}_r, q^*\hat{\mm}_0, \ldots, q^*\hat{\mm}_s$ are integrable on $X \times Y$, and the equality
\[ \langle p^*\hat{\ll}_0, \ldots, p^*\hat{\ll}_r, q^*\hat{\mm}_0, \ldots, q^*\hat{\mm}_s | X \times Y \rangle = \langle \hat{\ll}_0, \ldots, \hat{\ll}_r | X \rangle \langle \hat{\mm}_0, \ldots, \hat{\mm}_s | Y\rangle    \]
holds in $\rr$. 
\end{prop}
\begin{proof} Without loss of generality we may assume that the $\hat{\ll}_j$ resp. $\hat{\mm}_j$ are semipositive and approximated by semipositive models $(\xx_i,\bar{\ll}_{ij})$ of $(X,\ll_j^{\otimes e_{ij}})$ resp. $(\yy_i,\bar{\mm}_{ij})$ of $(Y,\mm_j^{\otimes f_{ij}})$. 
Denote by $\tilde{p}_i \colon \xx_i \times_S \yy_i \to \xx_i$ and $\tilde{q}_i \colon \xx_i \times_S \yy_i \to \yy_i$ the canonical projections extending $p$ and $q$. 
We claim that for each $i=1,2,\ldots$ the equality
\[ \langle \tilde{p}_i^*\bar{\ll}_{i0}, \ldots,  \tilde{q}_i^*\bar{\mm}_{is} | \xx_i \times_S \yy_i \rangle = \langle \bar{\ll}_{i0}, \ldots, \bar{\ll}_{ir} | \xx_i \rangle \langle \bar{\mm}_{i0}, \ldots, \bar{\mm}_{is} | \yy_i\rangle  \]
holds. Upon dividing both sides of the equality by $e_{i0}\cdots e_{ir}f_{i0}\cdots f_{is}$ and taking limits as $i \to \infty$ we then obtain the equality from the proposition. To prove the claim, recall that $\xx_i$ and $\yy_i$ can be realized as integral closed subschemes of projective schemes $\xx'_i$ and $\yy'_i$ that are smooth over $S$, and that the hermitian line bundles $\bar{\ll}_{ij}$ and $\bar{\mm}_{ij}$ can be assumed to be restrictions to $\xx_i$ and $\yy_i$ of hermitian line bundles on $\xx'_i$ and $\yy'_i$. The required equality then follows by an application of \cite[Proposition 2.3.3]{bgs}.
\end{proof}
\begin{remark} Both sides of the equality in Proposition \ref{externalproducts} vanish unless $\dim X=r+1$ and $\dim Y=s+1$, or $\dim X=r$ and $\dim Y=s+1$, or $\dim X=r+1$ and $\dim Y=s$.
\end{remark}
Our next goal is to introduce and study the Deligne pairing for integrable adelic line bundles. We refer to \cite[Section 1.1]{zhss} and \cite[Section 2.1]{zhgs} for more details. Let $X, Y$ be integral schemes and let $f \colon Y \to X$ be a flat projective morphism of relative dimension~$n$, and let $\mm_0,\ldots,\mm_n$ be $n+1$ line bundles on $Y$. Then we have a canonical line bundle $\langle \mm_0,\ldots,\mm_n \rangle$ on $X$ locally generated by sections $\langle m_0,\ldots,m_n \rangle$, where $m_0,\ldots,m_n$ are local sections of $\mm_0,\ldots,\mm_n$ with empty common zero locus, and with relations 
\[ \langle m_0, \ldots, h\, m_j , \ldots, m_n \rangle = \prod_s \mathrm{Nm}_{Z_s/X}(h)^{a_s} \langle m_0, \ldots,  m_n \rangle \]
for all $j=0,\ldots,n$ and all regular functions $h$ such that $\cap_{i \neq j} \divisor m_i=\sum_s a_s Z_s$ is finite flat over $X$ and has empty intersection with $\divisor h$. Here $\mathrm{Nm}_{Z_s/X}(h)$ denotes the norm of $h|_{Z_s}$ along the finite flat map $Z_s \to X$. The formation of the Deligne pairing $\langle \mm_0,\ldots,\mm_n \rangle$ is symmetric and multilinear in $\mm_0,\ldots,\mm_n$ and commutes with arbitrary base change. 

Now let $f \colon Y \to X$ be a flat morphism of relative dimension $n$ between two smooth projective varieties over $k$, let $\xx$ and $\yy$ be projective flat models of $X$ and $Y$ over the ring of integers $R$ of $k$, and assume that $f$ extends into a flat morphism $\tilde{f} \colon \tilde{Y} \to \tilde{X}$. Let $\bar{\mm}_0,\ldots,\bar{\mm}_n$ be $n+1$ smooth hermitian line bundles on $\yy$ with underlying line bundles $\tilde{\mm}_0,\ldots,\tilde{\mm}_n$ and generic fibers $\mm_0,\ldots,\mm_n$. Then for each $v \in M(k)_\infty$ the Deligne pairing $\langle \tilde{\mm}_0,\ldots,\tilde{\mm}_n \rangle$ along $\tilde{f}$ is equipped with a canonical hermitian metric $\|\cdot\|_v$ which is given recursively as follows. Let $m_0,\ldots,m_n$ be non-zero local sections of $\mm_0,\ldots,\mm_n$ over $Y$ with empty common zero locus. Assume that $\divisor m_n$ is a prime divisor on $Y$ and that $m_i|_{\divisor m_n}$ is non-zero for $i=0,\ldots,n-1$. Then for all $v \in M(k)_\infty$ we have 
\begin{equation} \label{local_recursive_infinite} \begin{split} \log \| \langle m_0,\ldots,m_n \rangle \|_v  = & \,\,
\log \| \langle m_0|_{\divisor m_n},\ldots,m_{n-1}|_{\divisor m_n} \rangle \|_v \\ & \,\,+ \int_{Y_v(\cc)/X_v(\cc)} \log \|m_n\|_v \, c_1(\bar{\mm}_{0,v}) \cdots c_1(\bar{\mm}_{n-1,v}) \, .   
\end{split} \end{equation}
By \cite[Theorem A]{mor} the metric $\|\cdot\|_v$ is indeed continuous on $X_v(\cc)$. If $f$ is smooth then the metric $\|\cdot\|_v$ is actually smooth on $X_v(\cc)$ by \cite[Proposition~8.5]{de} and \cite[Th\'eor\`eme I.1.1]{elkik}. 
\begin{prop} \label{pushforwardDeligne} Let $X$, $Y$ be smooth projective varieties over $k$, and let $f \colon Y \to X$ be a smooth morphism of relative dimension $n$. Let $\xx$ and $\yy$ be projective flat models of $X$ and $Y$ over $R$, and assume that $f$ extends into a flat morphism $\tilde{f} \colon \tilde{Y} \to \tilde{X}$. Then the Deligne pairing $(\langle \tilde{\mm}_0,\ldots,\tilde{\mm}_n \rangle, (\|\cdot\|_v)_{v \in M(k)_\infty})$ along $\tilde{f}$ is a smooth hermitian line bundle on $\xx$. Moreover, we have a natural induced map $\tilde{f}_* \colon \widehat{CH}^{n+1}(\yy) \to \widehat{CH}^1(\xx)$, and the equality
\[ \tilde{f}_*(\hat{c}_1(\bar{\mm}_0) \cdots \hat{c}_1(\bar{\mm}_n)) = \hat{c}_1(\langle \tilde{\mm}_0,\ldots,\tilde{\mm}_n \rangle, (\|\cdot\|_v)_{v \in M(k)_\infty})  \]
holds in $\widehat{CH}^1(\xx)$.
\end{prop}
\begin{proof} The first assertion rephrases what was said just before the proposition. The existence of $\tilde{f}_* \colon \widehat{CH}^{n+1}(\yy) \to \widehat{CH}^1(\xx)$ follows from \cite[Theorem 3.6.1]{giso}. To prove the equality, let $\tilde{m}_0,\ldots,\tilde{m}_n$ be sufficiently general rational sections of $\tilde{\mm}_0,\ldots,\tilde{\mm}_n$ over $\yy$, and write $m_0,\ldots,m_n$ for their restrictions to $Y$. For each $j=0,\ldots,n$ and each $v \in M(k)_\infty$ we set $g_{j,v}= -\log \|m_j\|^2_v$. Then for each $j=0,\ldots,n$ the class $\hat{c}_1(\bar{\mm}_j)$  in $\widehat{CH}^1(\yy)$ is represented by the Arakelov divisor $\left(\divisor \tilde{m}_j,g_{j,v}\right)$, and the class $\tilde{f}_*(\hat{c}_1(\bar{\mm}_0) \cdots \hat{c}_1(\bar{\mm}_n))$ in $\widehat{CH}^1(\xx)$ is represented by the Arakelov divisor 
\[ \left(\tilde{f}_*(\divisor \tilde{m}_0) \cdots \tilde{f}_*(\divisor \tilde{m}_n),   \int_{Y_v(\cc)/X_v(\cc)}
g_{0,v} \star \cdots \star g_{n,v}\right) \, , \]
using the star-product of Green's currents as in \cite[Section~2.1]{giso}.
On the other hand the class $\hat{c}_1(\langle \tilde{\mm}_0,\ldots,\tilde{\mm}_n \rangle)$ in $\widehat{CH}^1(\xx)$ is represented by the Arakelov divisor 
\[ \left(\divisor \langle \tilde{m}_0,\ldots, \tilde{m}_n \rangle , g_v  \right)   \]
where $g_v =-\log \|\langle m_0,\ldots,m_n \rangle\|^2_v$ for each $v \in M(k)_\infty$. Now we have 
\[ \divisor \langle \tilde{m}_0,\ldots, \tilde{m}_n \rangle=\tilde{f}_*(\divisor \tilde{m}_0) \cdots \tilde{f}_*(\divisor \tilde{m}_n) \, , \]
and the recursive relations (\ref{local_recursive_infinite}) imply that 
\[ \begin{split} g_v & = -\log \|\langle m_0,\ldots,m_n \rangle\|^2_v \\
& = - \sum_{i=0}^n \int_{Y_v(\cc)/X_v(\cc)}   \log \|m_i\|^2 \,\delta_{\divisor m_{i+1}} \wedge \cdots \wedge \delta_{\divisor m_n} \wedge
c_1(\mm_0) \cdots c_1(\mm_{i-1}) \\
& = \sum_{i=0}^n \int_{Y_v(\cc)/X_v(\cc)}   g_{i,v} \, \delta_{\divisor m_{i+1}} \wedge \cdots \wedge \delta_{\divisor m_n} \wedge
c_1(\mm_0) \cdots c_1(\mm_{i-1}) \\
& = \int_{Y_v(\cc)/X_v(\cc)} g_{0,v} \star \cdots \star g_{n,v} \, . 
\end{split} \]
The proposition follows.
\end{proof}
Let $\hat{\mm}_0,\ldots,\hat{\mm}_n$ be $n+1$ semipositive adelic line bundles on $Y$, approximated by semipositive models $(\yy_i,\bar{\mm}_{ij})$ of $(Y,\mm_j^{\otimes e_i})$ for $i=1,2,\ldots$. Let $\tilde{X}_i$ be any sequence of models of $X$. Possibly after replacing each $\yy_i$ by a birational model $\yy'_i \to \yy_i$ we may assume that there exist morphisms $\tilde{f}_i \colon \yy_i \to \xx_i$ extending $f$. By an elementary argument using Hilbert schemes (cf. \cite[Section 5.2]{rg}) we may even assume, after replacing $\tilde{X}_i$ by a suitable blow-up and replacing $\tilde{Y}_i$ by its strict transform along that blow-up, that each $\tilde{f}_i$ is flat. We then find  a sequence of Deligne pairings $\langle \tilde{\mm}_{i0},\ldots,\tilde{\mm}_{in} \rangle$ along $\tilde{f}_i$ on the models $\xx_i$. It turns out that this sequence defines an adelic line bundle
\[ \langle \hat{\mm}_0,\ldots,\hat{\mm}_n \rangle= (\langle \mm_0,\ldots,\mm_n \rangle, (\|\cdot\|_v)_{v \in M(k)}) \]
on $X$, independent of choices. Assuming that the morphism $f$ is smooth, the metrics $\|\cdot\|_v$ for $v \in M(k)_\infty$ are smooth by Proposition \ref{pushforwardDeligne}, and the adelic line bundle $\langle \hat{\mm}_0,\ldots,\hat{\mm}_n \rangle$ is then in fact integrable. 

The metrics $\|\cdot \|_v$ can be described recursively as follows. Assume without loss of generality that $X = \Spec k$ and that the models $\yy_i$ are normal. Let $m_0,\ldots,m_n$ be non-zero local sections of $\mm_0,\ldots,\mm_n$ with empty common zero locus.  Assume that $\divisor m_n$ is a prime divisor on $Y$ and that $m_i|_{\divisor m_n}$ is non-zero for $i=0,\ldots,n-1$. Then for all $v \in M(k)$ we have 
\begin{equation} \label{local_recursive} \begin{split} \log \| \langle m_0,\ldots,m_n \rangle \|_v  = & \,\,
\log \| \langle m_0|_{\divisor m_n},\ldots,m_{n-1}|_{\divisor m_n} \rangle \|_v \\ & \,\,+ \int_{Y(\bar{k}_v)} \log \|m_n\|_v \, c_1(\hat{\mm}_0) \cdots c_1(\hat{\mm}_{n-1}) \, .   
\end{split} \end{equation}
For $v\in M(k)_\infty$ this formula boils down to (\ref{local_recursive_infinite}). Assume therefore that $v \in M(k)$ is a finite place. Then we have to make sense of the integral  in (\ref{local_recursive}). We can do that using the models $(\yy_i,\bar{\mm}_{ij})$ of $(Y,\mm_j^{\otimes e_i})$ as follows. For all $i=1,2,\ldots$ let $\tilde{m}_{in}$ be the rational section of $\tilde{\mm}_{in}$ extending the section $m_n^{\otimes e_i}$ of $\mm_n^{\otimes e_i}$. Then we have on each $\yy_i$ an equality of Weil divisors
\[ \divisor \tilde{m}_{in} = e_i \, \overline{ \divisor{m_n} } + V_i \, ,\]
where $V_i=\sum_{v \in M(k)_0} V_{i,v}$ is a Weil divisor supported in the closed fibers $\yy_{i,v}$ of $\yy_i$ over $\xx_i=\Spec R$. Here $\overline{ \divisor{m_n} }$ denotes the Zariski closure of $\divisor{m_n}$ in $\yy_i$. Then we define
\[ \begin{split} \int_{Y(\bar{k}_v)} \log \|m_n\|_v \,& c_1(\hat{\mm}_0) \cdots c_1(\hat{\mm}_{n-1}) = \\
& \lim_{i \to \infty} c_1(\tilde{\mm}_{i0})\cdots c_1(\tilde{\mm}_{i,n-1})[V_{i,v}] \log Nv /e_i^n\, . \end{split}  \]
We see that the definition of $\langle \hat{\mm}_0,\ldots,\hat{\mm}_n \rangle$ extends to any collection of integrable line bundles $\hat{\mm}_0,\ldots,\hat{\mm}_n$ on $Y$. Also we note that in the context of the Berkovich analytic space $Y_v^{\mathrm{an}}$ associated to $Y_v$, where $v \in M(k)_0$, the measure $c_1(\hat{\mm}_0) \cdots c_1(\hat{\mm}_{n-1})$ would be the Chambert-Loir measure \cite{clmeas} associated to the integrable line bundles $\hat{\mm}_0,\ldots,\hat{\mm}_{n-1}$.

Continuing with the case that $X = \Spec k$, we note the global equalities
\[  \langle \hat{\mm}_0, \ldots, \hat{\mm}_n | Y \rangle  = \langle \langle \hat{\mm}_0, \ldots, \hat{\mm}_n \rangle | \Spec k \rangle = -\sum_{v \in M(k)} \log \| \langle m_0,\ldots,m_n \rangle \|_v   \]
in $\rr$, where $m_0,\ldots,m_n$ can be any non-zero rational sections of $\mm_0,\ldots,\mm_n$ whose supports have empty common intersection on $Y$.
Assuming that $\divisor m_n$ is a prime divisor on $Y$, we then obtain from the local recursive formula (\ref{local_recursive}) the global recursive formula
\begin{equation} \label{global_recursive} \begin{split} \langle \hat{\mm}_0, \ldots, \hat{\mm}_n | Y \rangle = &\,\, \langle \hat{\mm}_0, \ldots, \hat{\mm}_{n-1} | \divisor m_n \rangle \\ & - \sum_{v \in M(k)} \int_{Y(\bar{k}_v)} \log \|m_n\|_v \, c_1(\hat{\mm}_0) \cdots c_1(\hat{\mm}_{n-1}) \, .   
\end{split} \end{equation}
It is easily seen that one has (\ref{global_recursive}) for all non-zero rational sections $m_n$ of $\mm_n$.
\begin{prop}  \label{projII}
Let $X, Y$ be smooth projective varieties over $k$ and let $f \colon Y \to X$ be a smooth morphism of relative dimension $n$.  Let $\hat{\ll}_0,\ldots,\hat{\ll}_r$ be  integrable adelic line bundles on $X$, and let $\hat{\mm}_{0},\ldots,\hat{\mm}_n$ be integrable adelic line bundles on $Y$. Then $f^*\hat{\ll}_j$ are integrable adelic line bundles on $Y$, and we have that
\[ \langle f^*(\hat{\ll}_0), \ldots, f^*(\hat{\ll}_r),  \hat{\mm}_1, \ldots,  \hat{\mm}_n |Y \rangle = c_1(\mm_1) \cdots c_1(\mm_n)[f]
 \langle \hat{\ll}_0, \ldots, \hat{\ll}_r | X \rangle \, . 
\] 
Here 
$c_1(\mm_1) \cdots c_1(\mm_n)[f]$ denotes the multidegree of the generic fiber of $f \colon Y \to X$ with respect to the line bundles $\mm_1, \ldots, \mm_n$. We also have that $\langle \hat{\mm}_0, \ldots, \hat{\mm}_n\rangle$ is an integrable adelic line bundle on $X$, and 
the identity
\[ \langle f^*(\hat{\ll}_0), \ldots, f^*(\hat{\ll}_r), \hat{\mm}_0, \ldots,  \hat{\mm}_n |Y \rangle =  \langle \hat{\ll}_0, \ldots, \hat{\ll}_r, \langle \hat{\mm}_0, \ldots,  \hat{\mm}_n \rangle | X \rangle 
\] 
holds in $\rr$.
\end{prop}
\begin{proof} Without loss of generality we may assume that the $\hat{\ll}_j$ and $\hat{\mm}_j$ are semipositive, approximated  by semipositive models $(\xx_i,\bar{\ll}_{ij})$ of $(X,\ll_j^{\otimes e_i})$ and $(\yy_i,\bar{\mm}_i)$ of $(Y,\mm_j^{\otimes e'_i})$. We may assume in addition that there exist flat morphisms $\tilde{f}_i \colon \yy_i \to \xx_i$ for each $i=1,2,\ldots$ extending the morphism $f$. Since the morphism $f$ is smooth of relative dimension $n$ by assumption we have by \cite[Theorem 3.6.1]{giso} natural induced pushforward maps $\tilde{f}_{i*} \colon \widehat{\CH}^{n+1-s}(\tilde{Y}_i) \to \widehat{\CH}^{1-s}(\xx_i)$ for $s=0, 1$ and $i=1,2,\ldots$. Then by \cite[Theorem 4.4.3]{giso} we have for $s=0, 1$ and $i=1,2,\ldots$ an identity
\begin{equation} \label{bigformula} \begin{split} \tilde{f}_{i*} ( \tilde{f}_i^*(\hat{c}_1(\bar{\ll}_{i0})) & \cdots  
\tilde{f}_i^*(\hat{c}_1(\bar{\ll}_{ir}))\cdot \hat{c}_1( \bar{\mm}_{is}) \cdots \hat{c}_1(\bar{\mm}_{in} )) \\ & = 
\hat{c}_1(\bar{\ll}_{i0}) \cdots \hat{c}_1(\bar{\ll}_{ir}) \cdot  \tilde{f}_{i*}( \hat{c}_1(\bar{\mm}_{is}) \cdots \hat{c}_1(\bar{\mm}_{in}) ) \end{split} 
\end{equation}
in the arithmetic Chow group $\widehat{\CH}^{r+2-s}(\xx_i)$. We note that the condition in \cite[Theorem 4.4.3]{giso} that the models $\xx_i$ and $\yy_i$ should be regular can be removed in our setting by the argument in \cite[Section 2.3.1, Remark (ii)]{bgs}. Taking $s=1$ in (\ref{bigformula}) we find
\[ \langle \tilde{f}_i^*(\bar{\ll}_{i0}), \ldots,  
\tilde{f}_i^*(\bar{\ll}_{ir}), \bar{\mm}_{i1}, \ldots, \bar{\mm}_{in}|\yy_i \rangle = c_1(\mm_1) \cdots c_1(\mm_n)[f]
\langle \bar{\ll}_{i0} \cdots \bar{\ll}_{ir}|\xx_i \rangle   \]
in $\rr$. By Proposition \ref{pushforwardDeligne} the class of the Deligne pairing  $\langle \bar{\mm}_{i0},\ldots, \bar{\mm}_{in} \rangle$ along $\tilde{f}_i$ in $\widehat{CH}^1(\xx_i)$ is equal to $\tilde{f}_{i*}(\hat{c}_1(\bar{\mm}_{i0})\cdots\hat{c}_1(\bar{\mm}_{in}))$. Hence by taking $s=0$ in (\ref{bigformula}) we obtain the identity
\[ \langle \tilde{f}_i^*(\bar{\ll}_{i0}), \ldots,  
\tilde{f}_i^*(\bar{\ll}_{ir}),  \bar{\mm}_{i0}, \ldots, \bar{\mm}_{in}|\yy_i \rangle = 
\langle \bar{\ll}_{i0}, \ldots, \bar{\ll}_{ir},  \langle \bar{\mm}_{i0},\ldots, \bar{\mm}_{in} \rangle |\xx_i \rangle
\]
in $\rr$. We obtain the identities from the proposition by dividing both sides in the identities above by $e_i^{r+1} {e'_i}^{n}$ resp. $e_i^{r+1} {e'_i}^{n+1}$ and taking limits as $i \to \infty$.
\end{proof}

\section{Admissible bundles on abelian varieties and curves} \label{admissible}
 
We continue to work with the number field $k$. We recall that we have fixed a collection of absolute values $|\cdot|_v$ on all $k_v$ for $v \in M(k)$ satisfying the product formula.
Let $A$ be an abelian variety over $k$ and let $\ll$ be a symmetric  line bundle on $A$.
The following proposition follows directly from \cite[Theorem~2.2]{zhsmall}.
\begin{prop}  \label{basicadmissible} Let $\phi \colon 
\ll^{\otimes 4} \isom [2]^*\ll$ be an isomorphism of line bundles. Then there exists a unique integrable adelic metric $(\|\cdot \|_v)_v$ on $\ll$ such that for each $v \in M(k)$ the isomorphism $\phi$ is an isometry. If $\phi$ is changed into $a\phi$ for some $a \in k^*$ then for all $v \in M(k)$ the metric $\|\cdot\|_v$ changes into $|a|_v^{1/3}\|\cdot\|_v$.  
\end{prop}
We call $(\|\cdot \|_v)_v$ the admissible metric associated to $\phi$. We deduce from Proposition \ref{basicadmissible} that if $\ll$ is ample, the associated height $\h'_{\hat{\ll}}$ is independent of the choice of $\phi$. We therefore just denote this height by $\h'_\ll$. Let $Z$ be a closed subvariety of $A$. Then the height $\h'_\ll(Z)$ coincides with the N\'eron-Tate height of $Z$ as defined in \cite{gu} and \cite{phil} as explained in \cite[Section~3.1]{zhsmall}. In particular, for all $n \in \zz_{\geq 1}$ we have $\h'_\ll([n](Z))=n^2\h'_\ll(Z)$, and moreover we have the fundamental inequality $\h'_\ll(Z) \geq 0$. We mention that for $v \in M(k)_\infty$ the smooth hermitian line bundles $(\ll_v,\|\cdot\|_v)$ have a translation-invariant curvature form.

For the remainder of this section, let $X$ be a smooth projective connected curve defined over the number field $k$, and let $(J,\lambda)$ be the jacobian of $X$. We will work with a symmetric line bundle $\ll$ on $J$ defining the principal polarization $\lambda$. We will assume moreover that $\ll$ is rigidified at the origin. The rigidification determines a unique choice of isomorphism $\phi \colon \ll^{\otimes 4} \isom [2]^*\ll$ and hence a canonical structure of admissible adelic line bundle $\hat{\ll}$ on $\ll$. The construction of $\ll$ (which depends on the choice of a semi-canonical divisor $D$ on $X$) goes as follows. 

For every $D \in \mathrm{Div}^{g-1} X$ we first consider the line bundle
\[ \ll_D = \oo_J(t_{[D]}^*\Theta) \otimes e^*\oo_J(t_{[D]}^*\Theta)^{\otimes -1} \]
on $J$, where $\Theta \subset \mathrm{Pic}^{g-1} X$ is the canonical theta divisor of $X$ consisting of the classes of effective divisors, where $t_{[D]} \colon J=\mathrm{Pic}^0 X \to \mathrm{Pic}^{g-1} X$ is the translation along the divisor class $[D]$, and $e \in J(k)$ is the origin of $J$. Note that each line bundle $\ll_D$ is canonically rigidified at the origin. Moreover each $\ll_D$ is ample and induces the principal polarization $\lambda \colon J \to \hat{J}$ of $J$. Assume that $D$ is a semi-canonical divisor on $X$. Then the associated $\ll_D$ is symmetric, and we will take $\ll$ to be $\ll_D$.

For each $\alpha \in \mathrm{Div}^1 X$ recall the morphism $f_{1,\alpha} \colon X \to J$ given by sending $x \in X$ to the class of $x-\alpha$. For every line bundle $\mm$ on $X$ there exist (possibly after replacing $k$ by a finite extension) integers $e, e'$ and a divisor $\alpha \in \mathrm{Div}^1 X$ such that $f_{1,\alpha}^*\ll^{\otimes e'}$ is isomorphic, as a line bundle, to $\mm^{\otimes e}$ (in fact, note that $f_{1,\alpha}^*\ll$ is isomorphic to the line bundle $\oo(\alpha+D)$ on $X$). Given such $e, e'$ and $\alpha$ together with an isomorphism $\mm^{\otimes e} \isom f_{1,\alpha}^*\ll^{\otimes e'}$, pullback along $f_{1,\alpha}$ and transporting structure through the isomorphism yields a structure of integrable adelic metric on $\mm$. 

Varying $e, e', \alpha$ and the isomorphism $\mm^{\otimes e} \isom f_{1,\alpha}^*\ll^{\otimes e'}$, we call the resulting metrics on $\mm$ admissible metrics. If $(\mm_0, (\|\cdot\|_{0,v})_v)$ and $(\mm_1, (\|\cdot\|_{1,v})_v)$ are admissible line bundles on $X$, then so is their tensor product $(\mm_0\otimes\mm_1, (\|\cdot\|_{0,v}\otimes\|\cdot\|_{1,v})_v)$. Moreover, the admissible metrics on the trivial line bundle $\oo$ are all constant.

Let $m \colon J \times J \to J$ be the group law of $J$, and let $p_i \colon J \times J \to J$ for $i=1, 2$ denote the projections onto the two factors. We set
\[ \hat{\bb} = m^*\hat{\ll} \otimes p_1^* \hat{\ll}^{\otimes -1} \otimes 
p_2^* \hat{\ll}^{\otimes -1} \, , \]
equipped with the integrable adelic metric determined by the pullback admissible adelic metrics from $\hat{\ll}$ along $m$ and the $p_i$. The line bundle $\hat{\bb}$ is canonically rigidified, and the rigidification is an adelic isometry.
\begin{prop} \label{delpairingbiext} Let $Y$ be a projective variety over the number field $k$. Let $\hat{\mm}_0, \hat{\mm}_1$ denote integrable adelic line bundles on $X \times Y$ that are of relative degree zero for the canonical projection $\pi \colon X \times Y \to Y$. 
Then the Deligne pairings $\pair{\hat{\mm}_0,\hat{\mm}_0}$ and 
$\pair{\hat{\mm}_0,\hat{\mm}_1}$ are integrable adelic line bundles on $Y$.
Assume that $\hat{\mm}_0, \hat{\mm}_1$ are fiberwise admissible. Let $f_i \colon Y \to J$ for $i=0,1$ be the canonical map determined by the line bundle $\mm_i$ and write $f=(f_0,f_1) \colon Y \to J \times J$. Then there exist  isometries 
\[ \pair{ \hat{\mm}_0, \hat{\mm}_1 }^{\otimes -1} \isom f^* \hat{\bb}  \quad  \textrm{and} \quad \pair{ \hat{\mm}_0,\hat{\mm}_0}^{\otimes -1} \isom f_0^* \hat{\ll}^{\otimes 2} \]
of integrable adelic line bundles on $Y$.
\end{prop}
\begin{proof} The first statement follows from Proposition \ref{projII}. As to the isometries, without loss of generality we may assume that $X$ has a rational point $a \in X(k)$. It suffices to consider the universal case where $Y=J \times J$ and $f$ is the identity map, and the line bundles $\mm_0, \mm_1$ are the universal line bundles on $X \times J \times J$ rigidified along $a$. 
Let $\pp$ be the Poincar\'e bundle on $J \times \hat{J}$, and consider $\hat{J}$ as the functor of rigidified line bundles on $J$ with vanishing Chern class. Let $T$ be a test scheme over $k$ and let $x, y \in J(T)$.  Then 
\[ \begin{split} (x,y)^*(\mathrm{id} \times \lambda)^*\pp & = (x,[t_y^*\ll \otimes \ll^{\otimes -1} \otimes y^* \ll^{\otimes -1}])^*\pp \\ & = x^* t_y^*\ll \otimes x^*\ll^{\otimes -1} \otimes y^* \ll^{\otimes -1} \\
& = (x+y)^* \ll \otimes x^*\ll^{\otimes -1} \otimes y^* \ll^{\otimes -1} \\
& = (x,y)^*\bb \, , \end{split} \]
so that we have a canonical isomorphism of line bundles $\bb \isom (\mathrm{id} \times \lambda)^*\pp$ on $J \times J$.
From \cite[Section 2.9]{metrperm} we obtain a canonical isomorphism of line bundles
\[ \pair{ \mm_0, \mm_1 }^{\otimes -1} \isom (\mathrm{id} \times \lambda)^*\pp  \]
over $J \times J$. We claim that the induced isomorphism 
$\pair{ \mm_0, \mm_1 }^{\otimes -1} \isom \bb$
can be taken to be an isometry of adelic line bundles. Note that the induced metrics on the squares $\hat{\mm}_0^{\otimes 2}$ and $\hat{\mm}_1^{\otimes 2}$ are again fiberwise admissible. It follows that we have isometries
\[ [2]^* \pair{ \hat{\mm}_0, \hat{\mm}_1 } \isom \pair{ \hat{\mm}_0^{\otimes 2} , \hat{\mm}_1^{\otimes 2} } \isom \pair{ \hat{\mm}_0, \hat{\mm}_1 }^{\otimes 4} \, . \]
We also have an isometry $[2]^* \hat{\bb} \isom \hat{\bb}^{\otimes 4}$, hence by uniqueness we obtain an isometry $ \pair{ \hat{\mm}_0, \hat{\mm}_1 }^{\otimes -1} \isom \hat{ \bb}$. To obtain an isometry $ \pair{ \hat{\mm}_0,\hat{\mm}_0}^{\otimes -1} \isom  \hat{\ll}^{\otimes 2}$ we specialize the isometry just established to the diagonal in $J \times J$. 
\end{proof}
Assume from now on that $X$ is geometrically connected and has semistable reduction over $k$. An alternative way to obtain adelic metrics on all line bundles $\mm$ on $X$ is described in the paper \cite{zhadm}. For example, for each $p \in X(k)$ the line bundle $\oo(p)$ is endowed with a canonical adelic metric $((\|\cdot\|_{\oo(p),v})_v)$ obtained by using Green's functions. More precisely, if $v \in M(k)_0$ let $R \colon X(\bar{k}_v) \to \Gamma_v$ denote the reduction map to the dual graph $\Gamma_v$ of the reduction of $X$ at $v$, and let $g_{\mu,v}$ denote Zhang's Green's function on $\Gamma_v$. Then for all $p\neq x \in X(\bar{k}_v)$ we have $\log \|1\|_{\oo(p),v}(x) =g_{\mu,v}(R(p),R(x))$.  If $v \in M(k)_\infty$ let $g_{\mu,v}$ denote the Arakelov-Green's function on the compact Riemann surface $X_v(\cc)$. Then for all $p\neq x \in X_v(\cc)$ we have $\log \|1\|_{\oo(p),v}(x) =g_{\mu,v}(p,x)$.

We shortly denote the resulting adelic line bundle on $X$ by $\hat{p}$. The dualizing sheaf $\omega$ is also endowed with a canonical adelic metric, denoted in this paper by $\hat{\omega}$. Importantly, for each $p \in X(k)$ the canonical residue map $p^* (\hat{\omega} \otimes \hat{p} ) \isom \oo_{\Spec k}$ is an isometry of adelic line bundles on $\Spec k$. It is stated in \cite[Section~4.7]{zhadm} and worked out in more detail in \cite{heinz} that the adelic metrics on line bundles on $X$ obtained by this second method are precisely the admissible metrics on line bundles on $X$ described above by using pullbacks of $\hat{\ll}$ from the jacobian $J$.

Let $\pi \colon X^3 \to X^2$ denote the projection onto the first two factors, and let $s_0, s_1 \colon X^2 \to X^3$ denote the tautological sections of this projection. These give rise to a canonical relative degree zero line bundle $\oo(s_0-s_1)$ on $X^3$ over $X^2$. We have canonical isomorphisms of line bundles
\[ \oo(2 \, \Delta) \otimes p_1^*\omega \otimes p_2^* \omega \isom \langle \oo(s_0-s_1), \oo(s_0-s_1) \rangle^{\otimes -1} \isom \delta^* \ll^{\otimes 2} \, , \]
where $\delta \colon X^2 \to J$ denotes the difference map $(x,y) \mapsto [x-y]$. From the given admissible structure on $\ll$ and the canonical admissible line bundle $\hat{\omega}$ we then obtain an integrable structure $\hat{\Delta}$ on $\oo(\Delta)$. As it turns out, this structure does not depend on the choice of semi-canonical divisor $D$, and hence is canonical. We refer to \cite[Section~3.5]{zhgs} for a slightly different approach.
By restriction to horizontal or vertical slices $\{p\} \times X$ or $X \times \{p\}$ of $X^2$, where $p \in X(k)$, the integrable line bundle $\hat{\Delta}$ yields the canonical admissible line bundle $\hat{p}$. Finally we have that the canonical isomorphism $\Delta^*\oo(\Delta) \isom \omega^{\otimes -1}$ given by adjunction induces an isometry 
\begin{equation} \label{adjunction} \Delta^* \hat{\Delta} \isom \hat{\omega}^{\otimes -1} 
\end{equation}
of adelic line bundles on $X$. 

\section{Local invariants}  \label{phi_explained}

In this section we define the local $\varphi$-invariants appearing in Theorem \ref{mainintro}, following \cite{zhgs}, and discuss some of their properties.  
We start by discussing the non-archimedean $\varphi$-invariant. A metrized graph is a compact connected metric space $\Gamma$ such that for each $x \in \Gamma$ there exist $n \in \zz_{\geq 1}$ and $\epsilon \in \rr_{>0}$ such that $x$ has an open neighborhood isometric to the star-shaped set
\[ \{ z \in \cc \, : \, z = t \mathrm{e}^{2\pi i k/n} \,\, \textrm{for some} \,\, 0 \leq t < \epsilon \,\, \textrm{and some} \,\, k \in \zz \} \]
endowed with the path metric. Let $\Gamma$ be a metrized graph. A divisor on $\Gamma$ is an element of $\zz^{(\Gamma)}$; note that a divisor has naturally a degree in $\zz$. 

Let $g \geq 1$ be an integer. A polarization on $\Gamma$ is a divisor $K \in \zz^{(\Gamma)}$ of degree $2g-2$. We then call $g$ the genus of the polarized metrized graph $\bar{\Gamma}=(\Gamma,K)$. When $\Gamma$ is a metrized graph, then from (\ref{crmeasure}) we have a canonical measure $\mu_{\mathrm{can}}$ on $\Gamma$. Assume that a polarization $K$ is given so that $\bar{\Gamma}=(\Gamma,K)$ has genus $g$. Then, following \cite{zhadm}, one considers the admissible measure $\mu_a$ on $\bar{\Gamma}$ given by the identity
\begin{equation} \label{canandadmmeasure} \mu_a = \frac{1}{2g} \left(\delta_K + 2\mu_{\mathrm{can}} \right) \, . 
\end{equation}
The Green's function $g_\mu$ associated to $\mu_a$ is the unique piecewise quadratic function on $\Gamma \times \Gamma$ uniquely determined by the two conditons
\begin{equation} \label{green} \Delta_y \, g_\mu(x,y) = \delta_x(y) - \mu_a(y) \, , \quad \int_\Gamma g_\mu(x,y) \, \mu_a(y) =0 
\end{equation}
for all $x \in \Gamma$. 
Let $\delta(\Gamma)$ denote the total length of $\Gamma$. Then we set 
\begin{equation} \label{phiinv} \varphi(\bar{\Gamma}) = -\frac{1}{4} \delta(\Gamma) + \frac{1}{4} \int_{\Gamma}
g_{\mu}(x,x)((10g+2) \mu_a - \delta_K) \, . 
\end{equation}
The invariant $\varphi(\bar{\Gamma})$ is closely related to the epsilon-invariant \cite{mosharp} \cite{zhadm} of $\bar{\Gamma}$, given by
\begin{equation} \label{epsinv} \vareps(\bar{\Gamma}) = \int_\Gamma g_\mu(x,x)((2g-2)\mu_a + \delta_K ) \, . 
\end{equation}
Now let $X$ be a smooth projective geometrically connected curve of genus $g \geq 1$ with semistable reduction over a number field $k$. Let $v \in M(k)_0$ and let $\Gamma_v$ be the dual graph of the reduction of $X$ at $v$. Then by 
\cite[Section~2.2]{cl} \cite[Section~2.1]{zhadm} the metrized graph $\Gamma_v$ comes with a canonical polarization, resulting in a polarized metrized graph $\bar{\Gamma}_v=(\Gamma_v,K_v)$ of genus $g$. We then set $\varphi(X_v)=\varphi(\bar{\Gamma}_v)$. Likewise we set $\vareps(X_v)=\vareps(\bar{\Gamma}_v)$. 

We have $\varphi(X_v)=0$ for $g=1$, and also $\varphi(X_v)=0$ if $X$ has good reduction at $v$. In general we have $\varphi(X_v) \geq 0$. In fact, we have a more precise lower bound. Let $\delta_0(X_v)$ denote the number of non-separating geometric double points on the reduction of $X$ at $v$, and let $\delta_i(X_v)$ for $i=1,\ldots,[g/2]$ denote the number of geometric double points on the reduction of $X$ at $v$ whose local normalization has two connected components, one of arithmetic genus $i$, and one of arithmetic genus $g-i$.
In  \cite{zhgs} Zhang conjectured, and proved in the special case of so-called elementary graphs, for $g \geq 2$ a lower bound
\[ \varphi(X_v) \geq  c(g) \delta_0(X_v) + \sum_{i=1}^{[g/2]} \frac{2i(g-i)}{g}\delta_i(X_v) \, ,  \]
where $c(g)$ is some positive constant depending only on $g$. 
Zhang's conjecture was proved in low genera $g=2, 3, 4$ by X. Faber \cite[Theorem~3.4]{fab}, and by Z. Cinkir \cite[Theorem~2.11]{ciinv} for arbitrary $g \geq 2$. Cinkir showed that one can even take $c(2)=\frac{1}{27}$ and $c(g)=\frac{(g-1)^2}{2g(7g+5)}$ for $g \geq 3$, but these values can most likely be improved. A list of the non-archimedean $\varphi$-invariants of all polarized metrized graphs in genus two is provided in \cite{djadm}, and for all polarized metrized graphs in genus three in \cite{ciadm}. The case of polarized metrized graphs arising as reduction graphs of semistable hyperelliptic curves is completely understood by Yamaki's work \cite{yags}. 

We next discuss the archimedean $\varphi$-invariant. This invariant has been independently introduced and studied by N. Kawazumi in \cite{kawhandbook} \cite{kaw} (for the invariant $a$ defined in \cite{kawhandbook} \cite{kaw} we have $2\pi \,a = \varphi$).
Let $C$ be a compact and connected Riemann surface of genus $g \geq 1$. Let $\phi_\ell$ for $\ell=1, 2,\ldots$ be the normalized real eigenforms of the Arakelov Laplacian \cite[p.~393]{fa} on  $C$ and let  
$\lambda_1\leq \lambda_2 \leq \ldots$ be its positive eigenvalues. Let $(\omega_1,\ldots,\omega_g)$ be an orthonormal basis for the hermitian inner
product $(\omega,\eta) \mapsto \frac{i}{2} \int_{C} \omega \, \bar{\eta}$ on the space of holomorphic $1$-forms on $C$.
Then we set
\[ \varphi(C) = \sum_{\ell=1}^\infty \frac{2}{\lambda_\ell} \sum_{m,n=1}^g  \left| 
\int_{C} \phi_\ell \,\omega_m \, \bar{\omega}_n \right|^2 \, . \]
This invariant is well-defined and independent of the choice of $(\omega_1,\ldots,\omega_g)$.  If $X$ is a smooth projective geometrically connected curve of genus $g \geq 1$ defined over a number field $k$, and $v \in M(k)_\infty$,  we shortly write $\varphi(X_v)$ for $\varphi(X_v(\cc))$.

Note that we clearly have $\varphi(C) \geq 0$. In fact we have $\varphi(C)=0$ if $g=1$, and $\varphi(C)>0$ else, cf. \cite[Corollary 1.2]{kaw} or \cite[Remark following Proposition~2.5.3]{zhgs}. A fast method to compute $\varphi(C)$ when $g=2$ is described in \cite{pio}.

Let $(J,\lambda)$ denote the jacobian of $C$. Let $\ll$ be a symmetric ample line bundle on $J$ inducing the polarization $\lambda$, endowed with an admissible metric $\|\cdot\|$ (that is, a smooth hermitian metric whose curvature form is translation-invariant). Let $s$ be any non-zero global section of $\ll$. Then following \cite{aut} we put  
\[ I(J,\lambda) = -\int_{J} \log \|s\| \, \mu + \frac{1}{2} \log \int_{J} \|s\|^2 \, \mu\, , \]
where $\mu$ is the Haar measure on $J$. We note that $I(J,\lambda)$ is independent of the choice of $\ll$ and of $s$, moreover we have $I(J,\lambda)>0$ by Jensen's inequality.
A fundamental identity relating $\varphi(C)$ with the Faltings delta-invariant $\delta_F(C)$ from \cite[p.~402]{fa} and the invariant $I(J,\lambda)$  has recently been found by R.~Wilms in his paper \cite{wi}. 
\begin{thm} \label{wilmsrewrite} (R. Wilms) Let $\delta_F(C)$ be the Faltings delta-invariant of the compact connected Riemann surface $C$. Let $\kappa_0=\log(\pi \sqrt{2})$. Then the identity
\[ \delta_F(C) - 4g\log(2\pi)= -12\,\kappa_0 g + 24\,I(J,\lambda)+ 2\,\varphi(C)  \]
holds.
\end{thm}
\begin{proof} This is the main result of \cite{wi}, but slightly rewritten. Let $\theta$ be the Riemann theta-function of $(J,\lambda)$, seen as a global section of a suitable symmetric ample line bundle $\ll$ representing $\lambda$. Endow $\ll$ with its standard admissible metric giving $\theta$ the norm $\|\theta\|$ defined on \cite[p.~401]{fa}. Let then $ \log \|H\|(J,\lambda) = \int_{J} \log \|\theta\|\, \mu $. Then \cite[Theorem~1.1]{wi} states that
\begin{equation} \label{wilmss} \delta_F(C) = -24\,\log\|H\|(J,\lambda) + 2\,\varphi(C)-8g\log(2\pi) \, . 
\end{equation}
From \cite[Proposition~2.5.6]{bl} it follows that $ \int_{J} \|\theta\|^2 \, \mu = 2^{-g/2}$.
This gives
\begin{equation} \label{IandH} -\log\|H\|(J,\lambda) = I(J,\lambda) + \frac{g}{4}\log 2 \, . 
\end{equation}
The required identity follows by combining (\ref{wilmss}) and (\ref{IandH}).
\end{proof}
Now, let $X$ be a smooth projective geometrically connected curve of genus $g \geq 1$ with semistable reduction over a number field $k$. Let $(J,\lambda)$ denote its jacobian. We have the following arithmetic consequence of Wilms's result.
\begin{cor}  \label{corofwilms}  Let $\h_F(J)$ denote the stable Faltings height of $J$, and let $\pair{\hat{\omega},\hat{\omega}}$ denote the self-intersection of the canonical admissible relative dualizing sheaf on $X$. Then the identity 
\[ \begin{split}  [k:\qq] \, \h_F(J) = & \,\,\frac{1}{12} \pair{\hat{\omega},\hat{\omega}} + \frac{1}{12} \sum_{v \in M(k)_0} (\delta(X_v)+\vareps(X_v))\log Nv - \kappa_0 g \, [k:\qq]\\ & + 2 \sum_{v \in M(k)_\infty} I(J_v,\lambda_v) + \frac{1}{6} \sum_{v \in M(k)_\infty} \varphi(X_v) \end{split} \]
holds.
\end{cor}
\begin{proof}  Let $\pair{\bar{\omega},\bar{\omega}}$ be the usual self-intersection of the relative dualizing sheaf of the minimal regular model of $X$ over the ring of integers of $k$. The Noether formula \cite[Theorem~6]{fa} \cite[Th\'eor\`eme~2.5]{mb} then reads
\[ \begin{split} 12\,[k:\qq] \, \h_F(J) = & \,\, \pair{\bar{\omega},\bar{\omega}} + \sum_{v \in M(k)_0} \delta(X_v)\log Nv \\ &+ \sum_{v \in M(k)_\infty}  \delta_F(X_v)   -4g\,[k:\qq]\,\log(2\pi)   \, .  \end{split}
\]
By \cite[Theorem 5.5]{zhadm} the non-archimedean epsilon-invariants give the difference between $\pair{\bar{\omega},\bar{\omega}}$ and $\pair{\hat{\omega},\hat{\omega}}$, more precisely we have
\[ \label{zhangeps} \pair{\hat{\omega},\hat{\omega}} = \pair{\bar{\omega},\bar{\omega}} - \sum_{v \in M(k)_0} \varepsilon(X_v) \log Nv \, . 
\]
We thus find that
\[ \begin{split} 12\,[k:\qq] \, \h_F(J) =& \,\, \pair{\hat{\omega},\hat{\omega}} + \sum_{v \in M(k)_0} (\delta(X_v)+\vareps(X_v))\log Nv \\ & + \sum_{v \in M(k)_\infty}  \delta_F(X_v)   -4g\,[k:\qq]\,\log(2\pi)   \, . \end{split}
\]
Combining with Theorem \ref{wilmsrewrite} we obtain the identity stated in the corollary.
\end{proof}
We obtain Theorem \ref{thetaheight} in the case of potentially everywhere good reduction.
\begin{cor} \label{realcatalyzer} Assume that $J$ has everywhere good reduction. Then the identity 
\[ 2g \,[k:\qq] \, \h'_\ll(\Theta) = \frac{1}{12} \pair{\hat{\omega},\hat{\omega}} + \frac{1}{6} \sum_{v \in M(k)} \varphi(X_v) \log Nv \]
holds.
\end{cor}
\begin{proof}  As we assume that $J$ has everywhere good reduction, Autissier's result (\ref{aut}) yields that
\[  \h_F(J) = 2g \, \h'_\ll(\Theta) - \kappa_0 g + \frac{2}{[k:\qq]} \sum_{v \in M(k)_\infty} I(J_v,\lambda_v) \log Nv \, . 
\]
On the other hand, applying Corollary \ref{corofwilms} we have
\[  [k:\qq] \, \h_F(J) = \frac{1}{12} \pair{\hat{\omega},\hat{\omega}} - \kappa_0 g\,[k:\qq] + 2 \sum_{v \in M(k)_\infty} I(J_v,\lambda_v) + \frac{1}{6} \sum_{v \in M(k)} \varphi(X_v)\log Nv \, . \]
We find the stated identity by combining these two equations.
\end{proof}
In Section \ref{proofautforjac} we will use this special case to prove Theorem \ref{thetaheight} in general.

\section{N\'eron-Tate height of a tautological cycle} \label{NTheightspecial}

In this section we derive from Proposition \ref{delpairingbiext} an explicit formula for the N\'eron-Tate height of a tautological cycle $Z_{m,\alpha}$ on a jacobian, in terms of admissible adelic intersection theory. Let $g \geq 1$ be an integer and let $X$ be a smooth projective geometrically connected curve of genus $g$ with semistable reduction over a number field $k$. Let $r$ be an integer satisfying $0 \leq r \leq g$, and let $m=(m_1,\ldots,m_r)$ be an $r$-tuple of non-zero integers and let $D$ be a divisor on $X$ with $\deg D =  \sum m_i $. Let $(J,\lambda)$ denote the jacobian of $X$.
Consider the map
\[ f =f_{m,D} \colon X^r \longrightarrow J \, , \quad (x_1,\ldots,x_r) \mapsto 
\left[\sum_{i=1}^r m_i x_i - D\right]  \, ,\]
and denote by $Z=Z_{m,D}$ the image of $f$. We note that the map $f$ is generically finite.  For $i,j=1,\ldots,r$ let $p_i \colon X^r \to X$ denote the projection on the $i$-th factor, and let $p_{ij} \colon X^r \to X^2$ denote the projection on the $i$-th and $j$-th factor. Let $\hat{\omega}$ denote the canonical line bundle $\omega$ of $X$ equipped with its canonical admissible adelic metric, and let $\hat{\Delta}$ denote the line bundle associated to the diagonal $\Delta$ on $X^2$ equipped with its canonical integrable metric, both as discussed at the end of Section \ref{admissible}. Denote by $\hat{D}$ the line bundle $\oo(D)$ on $X$ endowed with its canonical admissible metric. Write $\hat{\omega}_i = p_i^*\hat{\omega}$, $\hat{\Delta}_{ij}=p_{ij}^*\hat{\Delta}$ and $\hat{D}_i=p_i^*\hat{D}$. Let $\ll$ be a symmetric ample line bundle on $J$ defining the principal polarization $\lambda$.
\begin{thm} \label{main} Each of $\hat{\omega}_i$, $\hat{\Delta}_{ij}$ and $\hat{D}_i$ is an integrable adelic line bundle on $X^r$, and the identity
\[ \h'_\ll(Z) = \frac{\left\langle \left(\sum_{i=1}^r m_i^2 \hat{\omega}_i - 2 \sum_{i<j} m_im_j \hat{\Delta}_{ij} + 2 \sum_{i=1}^r m_i \hat{D}_i - \pair{\hat{D},\hat{D}} \right)^{r+1} \bigg| X^r \right\rangle   }{2^{r+1} \deg (f) \deg_\ll (Z) (r+1)[k:\qq]} 
\]
holds, where $\deg(f)$ is the degree of $f$.
\end{thm}
\begin{proof}
Denote by $\pi \colon X^{r+1} \to X^r$ the projection on the first $r$ coordinates, with tautological sections $s_i \colon X^r \to X^{r+1}$ given by $(x_1,\ldots,x_r) \mapsto (x_1,\ldots,x_r,x_i)$ for $i=1,\ldots,r$. We denote by $\hat{s}_i$ the adelic metrized line bundle determined by the image of $s_i$ obtained by taking the fiberwise canonical admissible metric along $\pi$. Note that $\hat{s}_i = p_{i,r+1}^* \hat{\Delta}$, which implies that the $\hat{s}_i$ are integrable.  Let $p_{r+1} \colon X^{r+1} \to X$ denote projection on the last coordinate. Consider the Deligne pairing 
\[ \hat{\mm} = \hat{\mm}_{m,D}=\left\langle \sum_{i=1}^r m_i \hat{s}_i - p_{r+1}^*\hat{D},\sum_{i=1}^r m_i \hat{s}_i - p_{r+1}^*\hat{D} \right\rangle^{\otimes -1} \]
along the projection $\pi \colon X^{r+1} \to X^r$. 
Note that $\hat{D}_i = s_i^*p_{r+1}^*\hat{D}$, and that $\pair{\hat{s}_i,\hat{s}_i}=\hat{\omega}_i^{\otimes -1}$ and $\pair{\hat{s}_i,\hat{s}_j} = \hat{\Delta}_{ij}$, canonically. Thus, by expanding brackets we find
\begin{equation} \label{longM} \hat{\mm} = \sum_{i=1}^r m_i^2 \hat{\omega}_i - 2 \sum_{i<j} m_im_j \hat{\Delta}_{ij} + 2 \sum_{i=1}^r m_i \hat{D}_i - \pair{\hat{D},\hat{D}} \, , 
\end{equation}
canonically, as integrable adelic line  bundles on $X^r$.
Let $\hat{\ll}$ denote the line bundle $\ll$, endowed with any admissible adelic metric. Then by Proposition \ref{delpairingbiext} we have an isometry $ \hat{\mm}  \isom f^* \hat{\ll}^{\otimes 2} $
of adelic line bundles.  By the projection formula from Proposition \ref{projection} we have 
\begin{equation} \begin{split} \label{rewriteheight}
\langle \hat{\mm}^{r+1} | X^r \rangle & = 2^{r+1} \langle \hat{\ll}^{r+1} | f_*(X^r) \rangle = \\ 
 & = 2^{r+1} \deg (f) \langle \hat{\ll}^{r+1} |  Z \rangle \\ 
 & = 2^{r+1} \deg (f) \deg_\ll (Z) (r+1) [k:\qq]\h'_\ll(Z) \, . \end{split} \end{equation}
We obtain the required formula by combining (\ref{longM}) and (\ref{rewriteheight}).
\end{proof} 
We illustrate (the proof of) Theorem \ref{main} with a few examples that already exist in the literature. First, take $m=0$ and $D$ any divisor of degree zero on $X$. Then Theorem~\ref{main} gives for the point $[D] \in J$ that
\begin{equation} \label{faltingshr} 2 \,[k:\qq] \,\h'_\ll([D]) = -\pair{\hat{D},\hat{D}} \, . 
\end{equation}
This is precisely the well-known Faltings-Hriljac formula or Hodge Index Theorem \cite[Theorem 4]{fa} \cite[Theorem 3.1]{hr} \cite[Section 5.4]{zhadm}. Here the adelic intersection product $\pair{\hat{D},\hat{D}}$ is taken on the curve $X$.

Assume that $g \geq 2$. Let $\alpha \in \Div^1 X$, and write $\hat{\beta}=2\hat{\alpha}-\pair{\hat{\alpha},\hat{\alpha}}$. Let $K_X$ denote the canonical divisor class of $X$.
Taking $D \in \mathrm{Div}^0 X$ a representative of the class $x_\alpha=\alpha - \frac{1}{2g-2}K_X$ we find from (\ref{faltingshr}), with all adelic intersection products taken on the curve $X$,
\begin{equation} \label{htpoint} \begin{split} 2\,[k:\qq]\,\h'_\ll(x_\alpha) & = -\pair{\hat{\alpha},\hat{\alpha}} + \frac{1}{g-1}\pair{\hat{\alpha},\hat{\omega}} - \frac{1}{(2g-2)^2}\pair{\hat{\omega},\hat{\omega}} \\
& = \frac{1}{2g-2}\pair{\hat{\beta},\hat{\omega}} - \frac{1}{(2g-2)^2}\pair{\hat{\omega},\hat{\omega}} \, .  \end{split} \end{equation}

Denote by $s$ the section $x \mapsto (x,x)$ of the projection on the first coordinate $p_1 \colon X^2 \to X$. Then for $\hat{\mm}=\hat{\mm}_{1,\alpha}$ we have
\[ \hat{\mm} = \pair{\hat{s}- p_2^*\hat{\alpha}, \hat{s} - p_2^*\hat{\alpha}}^{\otimes -1} = \hat{\omega}(2\hat{\alpha}) - \pair{\hat{\alpha},\hat{\alpha}} = \hat{\omega} + \hat{\beta} \, , \]
canonically, as integrable adelic line bundles on $X$.  We compute, noting that $\pair{\hat{\beta},\hat{\beta}} = 0$ and using (\ref{htpoint}),
\[ \begin{split} \hat{\mm}^2 & = \pair{\hat{\omega},\hat{\omega}} + 2 \pair{\hat{\beta},\hat{\omega}} \\
 &= \pair{\hat{\omega},\hat{\omega}} + \frac{1}{g-1}\pair{\hat{\omega},\hat{\omega}} + 8(g-1)\,[k:\qq]\,\h_\ll'(x_\alpha)
\\ & = \frac{g}{g-1}\pair{\hat{\omega},\hat{\omega}} + 8(g-1)\,[k:\qq]\, \h'_\ll(x_\alpha) \, . \end{split} \]
Theorem \ref{main} then gives, noting that $\deg_\ll Z_{1,\alpha}=g$,
\[ \begin{split} \h'_\ll(Z_{1,\alpha}) & = \frac{1}{8g} \left(\frac{g}{(g-1)[k:\qq]}\pair{\hat{\omega},\hat{\omega}} + 8(g-1) \h'_\ll(x_\alpha)\right) \\ & = \frac{1}{8(g-1)[k:\qq]} \pair{\hat{\omega},\hat{\omega}} + \frac{g-1}{g}\h'_\ll(x_\alpha) \, . \end{split} \]
This recovers \cite[Theorem~3.9]{zhsmall}. 

\begin{remark} \label{degrees} In general we have $\deg_\ll Z_{r,\alpha} = g!/(g-r)!$. Indeed, let $\Theta$ be a divisor with $\ll \cong \oo_J(\Theta)$, then in cohomology we have
\[ [Z_{r,\alpha}] = \frac{1}{(g-r)!} \bigwedge^{g-r} [\Theta] \]
by Poincar\'e's formula (see for instance \cite[Section~11.2]{bl}). It follows that
\[ \deg_\ll Z_{r,\alpha} = \frac{1}{(g-r)!} \deg_\ll \Theta \, . \]
We conclude by noting that $\deg_\ll \Theta= g!$ by the Riemann-Roch theorem.
\end{remark}

\section{Proof of Theorem \ref{mainintro} } \label{proofmain}

We continue with the notation from the previous section, and assume throughout that $g \geq 2$.
In order to prove Theorem \ref{mainintro} we will apply Theorem \ref{main} with $D=d\alpha$, where $d=\sum_{i=1}^r m_i$. As above, write $\hat{\beta}=2\hat{\alpha}-\pair{\hat{\alpha},\hat{\alpha}}$. Let $\hat{\beta}_i = p_i^*\hat{\beta}$. 
Following Theorem \ref{main}, our aim is to write
\begin{equation} \left\langle \left(\sum_{i=1}^r m_i^2 \hat{\omega}_i - 2 \sum_{i<j} m_im_j \hat{\Delta}_{ij} + d \sum_{i=1}^r m_i \hat{\beta}_i \right)^{r+1}\bigg| X^r \right\rangle 
\end{equation}
as a rational linear combination of $\pair{\hat{\omega},\hat{\omega}}$, $\sum_{v \in M(k)} \varphi(X_v)\log Nv$ and $[k:\qq]\h'_\ll(x_\alpha)$ with coefficients that only depend on $m$ and $g$. By equation (\ref{htpoint}) we may as well aim for a rational linear combination of $\pair{\hat{\omega},\hat{\omega}}$, $\sum_{v \in M(k)} \varphi(X_v)\log Nv$ and $\pair{\hat{\beta},\hat{\omega}}$ with coefficients that only depend on $m$ and $g$. By the following proposition, we may then as well aim for a rational linear combination of $\pair{\hat{\omega},\hat{\omega}}$, $\langle \hat{\Delta},\hat{\Delta},\hat{\Delta} \rangle$ and $\pair{\hat{\beta},\hat{\omega}}$ with coefficients that only depend on $m$ and $g$, where $\langle \hat{\Delta},\hat{\Delta},\hat{\Delta} \rangle$ is the triple self-product of the canonical integrable bundle $\hat{\Delta}$ on $X^2$.
\begin{prop} \label{triplediagonal} Let $\langle \hat{\Delta},\hat{\Delta},\hat{\Delta} \rangle$ denote the triple self-product of the integrable line bundle $\hat{\Delta}$ on $X^2$. Then the identity
\[ \langle \hat{\Delta}, \hat{\Delta}, \hat{\Delta} \rangle = \pair{\hat{\omega},\hat{\omega}} -  \sum_{v \in M(k)} \varphi(X_v) \log Nv \]
holds in $\rr$.
\end{prop}
\begin{proof} We apply the global recursive formula (\ref{global_recursive}) to the morphism $X^2 \to \Spec k$ and the section $1$ of $\oo(\Delta)$ on $X^2$. We find 
\[ \langle \hat{\Delta}, \hat{\Delta}, \hat{\Delta} \rangle  = 
\pair{\hat{\Delta}|_\Delta,\hat{\Delta}|_\Delta} - \sum_{v \in M(k)} \int_{X(\bar{k}_v)^2} \log \|1\|_{\Delta,v} \, c_1(\hat{\Delta})\,c_1(\hat{\Delta}) \, . \]
Now we recall from (\ref{adjunction}) that $\Delta^*\hat{\Delta} = \hat{\omega}^{\otimes -1}$ so that $\pair{\hat{\Delta}|_\Delta,\hat{\Delta}|_\Delta}= \pair{\hat{\omega},\hat{\omega}}$. 
By \cite[Proposition 2.5.3]{zhgs} for $v \in M(k)_\infty$ and \cite[Lemma 3.5.4]{zhgs} for $v \in M(k)_0$ we have 
\[ \int_{X(\bar{k}_v)^2} \log \|1\|_{\Delta,v} \, c_1(\hat{\Delta})\,c_1(\hat{\Delta}) =\int_{X(\bar{k}_v)^2} g_{\mu,v} \, c_1(\hat{\Delta})\,c_1(\hat{\Delta}) = \varphi(X_v)\log Nv \, . \]
The result follows.
\end{proof}
Consider the set
\[ \mathcal{S}(r) = \{ \hat{\omega}_1,\ldots,\hat{\omega}_r, \hat{\Delta}_{12},\ldots,\hat{\Delta}_{r-1,r},\hat{\beta}_1,\ldots,\hat{\beta}_r \}  \]
of integrable adelic line bundles on $X^r$.
We define a map $q \colon \mathcal{S} \to \zz$ by putting
\[ q(\hat{\omega}_i) = m_i^2 \, , \quad
q(\hat{\Delta}_{ij}) = -2m_im_j \, , \quad 
q(\hat{\beta}_i) = dm_i \]
for $i,j=1,\ldots, r$, $i < j$. We obtain
\[ \begin{split} 
\left\langle \left(\sum_{i=1}^r m_i^2 \hat{\omega}_i - 2 \sum_{i<j}  m_im_j \hat{\Delta}_{ij} + d \sum_{i=1}^r m_i \hat{\beta}_i \right)^{r+1}\bigg| X^r \right\rangle \\  = \sum_{ (\hat{\ll}_0,\ldots,\hat{\ll}_r) \in \mathcal{S}(r)^{r+1}} q(\hat{\ll}_0)\cdots q(\hat{\ll}_r) \langle
\hat{\ll}_0,\ldots,\hat{\ll}_r |X^r \rangle \, . \end{split} \] 
We claim that each $\langle
\hat{\ll}_0,\ldots,\hat{\ll}_r |X^r \rangle$ where $(\hat{\ll}_0,\ldots,\hat{\ll}_r) \in \mathcal{S}(r)^{r+1}$
is a rational multiple of $\pair{\hat{\omega},\hat{\omega}}$, $\langle \hat{\Delta},\hat{\Delta},\hat{\Delta} \rangle$ or $\pair{\hat{\beta},\hat{\omega}}$ with a coefficient that depends on a (decorated) graph associated to $(\hat{\ll}_0,\ldots,\hat{\ll}_r)$ (to be defined shortly) and which for fixed $g$ is universal in $(X,\alpha)$. This suffices to prove Theorem \ref{mainintro}.

We use a graphical calculus that is very similar to a method used in R. Wilms's article \cite{wi}. We refer to \cite{pand} \cite{rw} for variants of this calculus in the context of the tautological ring on the moduli space of curves. First of all, it is convenient to slightly generalize the set $\mathcal{S}(r)$ as defined above. Let $V \subset \nn$ be a finite set. We have natural projections $p_i \colon X^V \to X$ for $i \in V$ and $p_{ij} \colon X^V \to X^2$ for $i,j \in V$, $i<j$. We define the set
\[ \mathcal{S}(V) = \{ \hat{\omega}_i \, (i \in V) \, , \, -\hat{\omega}_i \, (i \in V) \, , \, \hat{\Delta}_{ij} \, (i,j \in V, i<j) \, , \, \hat{\beta}_i \, (i \in V) \}  \]
consisting of integrable line bundles on $X^V$ obtained by pulling back from $X$, $X^2$ along $p_i$, $p_{ij}$. 

We say that an undirected graph $\Gamma=(V,E)$ on $V$ (loops and multiple edges allowed) is loop-labelled if each loop $e$ in $E$ has been assigned an element $\ell(e) \in \{\hat{\omega}, -\hat{\omega},\hat{\beta} \}$. We say that $\Gamma$ is edge-ordered if some linear ordering of $E$ has been chosen. Let $\Gamma=(V,E)$ be a loop-labelled edge-ordered graph with edges $e_0,\ldots,e_n$. Then to $\Gamma$ we associate a tuple $(\hat{\ll}_0,\ldots,\hat{\ll}_n)\in \mathcal{S}(V)^{n+1}$ as follows: for each $k=0,\ldots,n$ we define the integrable line bundle $\hat{\ll}_k \in \mathcal{S}(V)$ to be the integrable line bundle $\hat{\psi}_i=p_i^* \hat{\psi}$ if $e_k$ is a loop based at $i \in V$ and $\ell(e_k)=\hat{\psi}$, and to be the integrable line  bundle $\hat{\Delta}_{ij}=p_{ij}^*\hat{\Delta}$ if $e_k$ is an edge between $i$ and $j$. Conversely, given a tuple $(\hat{\ll}_0,\ldots,\hat{\ll}_n) \in \mathcal{S}(V)^{n+1}$ one has naturally associated to it a loop-labelled edge-ordered graph with $n+1$ edges on $V$, that we denote by $\Gamma(\hat{\ll}_0,\ldots,\hat{\ll}_n)$. We define the intersection number $\mu(\Gamma)$ of a loop-labelled edge-ordered graph $\Gamma=(V,E)$ to be the arithmetic intersection number $\langle \hat{\ll}_0,\ldots, \hat{\ll}_n | X^V \rangle \in \rr$ if $(\hat{\ll}_0,\ldots,\hat{\ll}_n)\in \mathcal{S}(V)^{n+1}$ is the tuple associated to $\Gamma$. Note that the intersection number does not depend on the actual ordering of the edges, and hence yields an invariant of merely loop-labelled graphs.
\begin{prop} \label{multiplicative} The intersection number of a loop-labelled graph $\Gamma$ is multiplicative over connected components of $\Gamma$.
\end{prop}
\begin{proof} This follows immediately from Proposition \ref{externalproducts}.
\end{proof}

Let $\Gamma$ be a connected loop-labelled graph on a vertex set $V \subset \nn$. A contraction of $\Gamma$ is the connected loop-labelled graph obtained by performing one of the following operations on $\Gamma$: (i) contract an edge emanating from a vertex of degree one; (ii) if vertex $j$ of $\Gamma$ has degree $2$ with emanating edges $ij$ and  $jk$ with $i \neq k$, then replace the edges $ij$, $jk$ by one edge $ik$; (iii) if vertex $i$ of $\Gamma$ has degree $2$ with both emanating edges connecting $i$ and $j$, and $i \neq j$, then replace the two edges by one loop attached at $j$ and labelled $-\hat{\omega}$. 

\begin{prop} \label{contractions} The intersection number of a connected loop-labelled graph is invariant under contractions.
\end{prop}
\begin{proof} We consider each of the three types of contraction separately. After relabeling the vertices we may assume that $V=\{1,\ldots,r\}$. In case (i), after reordering the vertices we may assume that vertex $1$ has degree one and that the unique edge emanating from it is $12$. The following observation then proves the assertion in this case: let $p_{12} \colon X^r \to X^2$ be the projection on the first two coordinates, and let $p^1 \colon X^r \to X^{r-1}$ be the projection forgetting the first coordinate. Let $\hat{\mm}_1,\ldots,\hat{\mm}_n$ be integrable adelic line bundles on $X^{r-1}$. Then as $\deg_{\Delta_{12}}(p^1)=1$ we find by Proposition \ref{projII} that
\[ \langle p_{12}^*\hat{\Delta}, p^{1*}\hat{\mm}_1, \ldots, p^{1*}\hat{\mm}_n |X^r \rangle  =  \langle \hat{\mm}_1,\ldots,\hat{\mm}_n | X^{r-1} \rangle \, . 
\]
In case (ii) it suffices by Proposition \ref{projII} to verify the following: let $p_{13} \colon X^3 \to X^2$ be the projection on the first and third coordinate. Then for the Deligne pairing of $\hat{\Delta}_{12}$ and $\hat{\Delta}_{23}$ along $p_{13}$ one has an isometry $\pair{\hat{\Delta}_{12}, \hat{\Delta}_{23}} \isom \hat{\Delta}_{13}$ of adelic line bundles. This assertion follows from the recursive formula (\ref{global_recursive}) applied to $p_{13}$ and the divisor $\Delta_{13}$ on $X^3$. For each $v \in M(k)$ the integral in (\ref{global_recursive}) vanishes.

In case (iii), after reordering the vertices we may assume that vertex $1$ has degree two and that both edges emanating from it end at vertex $2$. The following observation then proves the assertion in this case: let $p_{12} \colon X^r \to X^2$ be the projection on the first two coordinates, and let $p^1 \colon X^r \to X^{r-1}$ be the projection forgetting the first coordinate. Let $\hat{\mm}_2,\ldots,\hat{\mm}_n$ be integrable adelic line bundles on $X^{r-1}$. Then by Proposition \ref{projII} we have
\[ \langle p_{12}^*\hat{\Delta}, p_{12}^*{\hat{\Delta}}, p^{1*} \hat{\mm}_2, \ldots, p^{1*} \hat{\mm}_n |X^r \rangle = \langle \langle p_{12}^*\hat{\Delta}, p_{12}^*{\hat{\Delta}} \rangle,  \hat{\mm}_2, \ldots \hat{\mm}_n |X^{r-1} \rangle \, , \]
where the Deligne pairing is along $p^1$.
Next we have $ \langle \hat{\Delta} , \hat{\Delta} \rangle = \Delta^*\hat{\Delta} $ by the recursive formula (\ref{global_recursive}) applied to a projection $X^2 \to X$ and the rational section $1$ of $\oo(\Delta)$ on $X^2$. Note that for each $v \in M(k)$ the integral in (\ref{global_recursive}) vanishes. From the adjunction formula (\ref{adjunction}) we then finally find that $ \langle \hat{\Delta} , \hat{\Delta} \rangle= -\hat{\omega}$.
\end{proof}

Observe that under a contraction both the number of edges and the number of vertices of a connected loop-labelled graph decrease by one. In particular, the Euler characteristic $\chi$ remains preserved, and after finitely many steps no more contractions are possible. A connected loop-labelled graph on which no contractions are possible is called minimal. A minimal graph either consists of one vertex (these are the minimal graphs with $\chi=1$), or of one vertex with one loop attached to it (these are the minimal graphs with $\chi=0$), or has each vertex of degree at least $3$ (and $\chi<0$). In the latter case, if we let $v$ denote the number of vertices of $\Gamma$, and $e$ the number of edges, then $3v \leq 2e = 2v-2\chi$ and hence $v \leq -2\chi$. In particular, if $\chi=-1$, then $v \leq 2$. 

We call a loop-labelled graph $\Gamma$ relevant if all connected components of $\Gamma$ have their Euler characteristic $\chi \in \{0,-1\}$, and $\chi=-1$ occurs for precisely one connected component. 

Recall that our task is to show that for each tuple $(\hat{\ll}_0 ,\ldots, \hat{\ll}_r )\in \mathcal{S}(r)^{r+1}$ the intersection number  $\langle \hat{\ll}_0, \ldots, \hat{\ll}_r |X^r \rangle$ is a rational multiple of
$\pair{\hat{\omega},\hat{\omega}}$, $\langle \hat{\Delta},\hat{\Delta},\hat{\Delta} \rangle$ or $\pair{\hat{\beta},\hat{\omega}}$ with a coefficient that for fixed $g$ is universal in $(X,\alpha)$. Note that the associated graph $\Gamma$ of such a tuple $(\hat{\ll}_0 ,\ldots, \hat{\ll}_r )$ has $r$ vertices and $r+1$ edges, hence has $\chi(\Gamma)=-1$. Our task would thus be accomplished once we show that the intersection number of any given loop-labelled graph $\Gamma$ with $\chi(\Gamma)=-1$ is a rational multiple of $\pair{\hat{\omega},\hat{\omega}}$, $\langle \hat{\Delta},\hat{\Delta},\hat{\Delta} \rangle$ or $\pair{\hat{\beta},\hat{\omega}}$ with a coefficient that only depends on the underlying graph of $\Gamma$ and the degrees in $\zz$ of the labels of $\Gamma$. 
Hence, we are done by the proposition below.
\begin{prop} Let $\Gamma$ be a loop-labelled graph with $\chi(\Gamma)=-1$. If $\Gamma$ is not relevant, the intersection number $\mu(\Gamma)$ vanishes. Assume $\Gamma$ is relevant. Then $\mu(\Gamma)$ is zero or a product of powers of $-1$, $2$ and $2g-2$ times one of the intersection numbers $\langle \hat{\Delta}, \hat{\Delta}, \hat{\Delta} \rangle, \langle \hat{\beta}, \hat{\omega} \rangle, 
\langle \hat{\omega}, \hat{\omega} \rangle$.
\end{prop}
\begin{proof} Assume that $\Gamma$ is not relevant, that is, some component of $\Gamma$ has $\chi \notin \{0,-1\}$. As $\chi(\Gamma)=-1$ there is in fact a component with $\chi=1$. This component is a tree, which is contractible to a point, and hence has vanishing intersection product. By Proposition \ref{multiplicative}, the intersection product of $\Gamma$ vanishes. Assume that $\Gamma$ is relevant, and let $\Gamma'$ be a connected component of $\Gamma$. By Proposition \ref{multiplicative}, it suffices to prove that $\mu(\Gamma')$ is either $0$, $2-2g$, $2$, $2g-2$ or plus or minus one of the intersection numbers $\langle \hat{\Delta}, \hat{\Delta}, \hat{\Delta} \rangle, \langle \hat{\beta}, \hat{\omega} \rangle, \langle \hat{\omega}, \hat{\omega} \rangle$.
By Proposition \ref{contractions}, we may assume without loss of generality
that $\Gamma'$ is minimal. If $\chi(\Gamma')=0$, then $\Gamma'$ consists of one vertex with one loop attached to it.  If $\hat{\sigma}$ is the label of the unique loop of $\Gamma'$ we have $\mu(\Gamma')=\deg \sigma$, in particular $\mu(\Gamma') \in \{2-2g,2,2g-2\}$. If $\chi(\Gamma')=-1$ then as we saw above $\Gamma'$ has at most two vertices and hence $\Gamma'$ is one of the following three loop-labelled graphs:
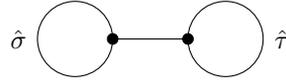
\begin{center}
\begin{tikzpicture}
\draw (-1,0) circle [radius=0.5];
\draw (1,0) circle [radius=0.5];
\draw (-0.5,0) node[circle, draw, fill=black, inner sep=0pt, minimum width=4pt]{} -- (0.5,0) node[circle, draw, fill=black, inner sep=0pt, minimum width=4pt]{};
\node at (-1.75,0) {$\hat{\sigma}$};
\node at (1.75,0) {$\hat{\tau}$};
\node [above] at (0,0) {};
\end{tikzpicture}
\end{center}
or
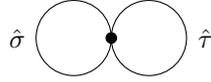
\begin{center}
\begin{tikzpicture}
\draw (-0.5,0) circle [radius=0.5];
\draw (0.5,0) circle [radius=0.5];
\draw (0,0) node[circle, draw, fill=black, inner sep=0pt, minimum width=4pt]{};
\node at (-1.25,0) {$\hat{\sigma}$};
\node at (1.25,0) {$\hat{\tau}$};
\end{tikzpicture}
\end{center}
or
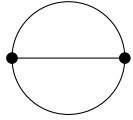
\begin{center}
\begin{tikzpicture}
\draw (0,0) circle [radius=0.75];
\draw (-0.75,0) node[circle, draw, fill=black, inner sep=0pt, minimum width=4pt]{} -- (0.75,0) node[circle, draw, fill=black, 
inner sep=0pt, minimum width=4pt]{};
\node at (0.75,0.75) {};
\node [above] at (0,0) {};
\node at (0.75,-0.75) {};
\end{tikzpicture}
\end{center}
where $\hat{\sigma}, \hat{\tau} \in \{  \hat{\omega}, -\hat{\omega}, \hat{\beta} \}$. In the first two cases, we have $\mu(\Gamma')=\langle \hat{\sigma},\hat{\tau} \rangle$ so that $\mu(\Gamma') \in \{ \pair{\hat{\beta},\hat{\omega}}, \pair{\hat{\omega},\hat{\omega}}, \pair{\hat{\beta},\hat{\beta}}\}$, up to sign. Note that $\pair{\hat{\beta},\hat{\beta}}=0$. In the third case, we have $\mu(\Gamma')=
\langle \hat{\Delta}, \hat{\Delta}, \hat{\Delta} \rangle$. This proves the proposition.
\end{proof}

\section{Surfaces} \label{surfaces}

Let $X$ be a smooth projective geometrically connected curve of genus $g \geq 2$ with semistable reduction over a number field $k$. Let $(J,\lambda)$ denote its jacobian. In this section we illustrate Theorem \ref{mainintro} by considering the Abel-Jacobi images of $X^2$ as well as the difference surface $F$ inside $J$.

Let $\alpha \in \mathrm{Div}^1 X$, and as before let $\hat{\alpha}$ be shorthand for the line bundle $\oo(\alpha)$ on $X$ endowed with its canonical admissible adelic metric. Recall that we put $\hat{\beta}=2\hat{\alpha}-\pair{\hat{\alpha},\hat{\alpha}}$. We shall frequently make use of the following observations, that easily follow from the graphical calculus discussed in the previous section. Let $\hat{\sigma}, \hat{\tau}, \hat{\psi}$ be any of the admissible line bundles $\hat{\omega}$, $\hat{\beta}$ on $X$, and let $p_1, p_2 \colon X^2 \to X$ denote the two projections. For $i=1,2$ let $\hat{\sigma}_i, \hat{\tau}_i, \hat{\psi}_i$ denote the pullbacks of $\hat{\sigma}, \hat{\tau}, \hat{\psi}$ along $p_i$. 
Then we have the relations 
\[ \langle \hat{\Delta}, \hat{\Delta},  \hat{\sigma}_i \rangle = -\pair{\hat{\omega},\hat{\sigma}}  \]
and
\[ \langle \hat{\Delta}, \hat{\sigma}_i, \hat{\tau}_j \rangle = \langle \hat{\sigma}, \hat{\tau} \rangle \, . \]
Next, a triple intersection 
\[ \langle \hat{\sigma}_i, \hat{\tau}_j,\hat{\psi}_k \rangle  \]
vanishes if $i=j=k$. On the other hand if $i=j$ and $k$ is different from $i$ and $j$ then we have 
\[ \langle \hat{\sigma}_i, \hat{\tau}_j,\hat{\psi}_k \rangle = \deg (\psi) \pair{\hat{\sigma},\hat{\tau}} 
\, . \]
Let $Z_{2,\alpha}$ be the image of $X^2$ in $J$ under the map $f_{2,\alpha} \colon X^2 \to J$. By Remark~\ref{degrees} we have $ \deg f_{2,\alpha} \deg_\ll Z_{2,\alpha} = 2g(g-1)$. Let $F =Z_{(1,-1)}\subset J$ be the image of $X^2$ under $f_{(1,-1)}$. The homology classes $f_{(1,-1)*}[X^2]$ and $f_{2,\alpha,*}[X^2]$ in $H_4(J)$ are equal, and hence $\deg f_{(1,-1)} \deg_\ll F =  2g(g-1)$ as well.
\begin{thm} \label{heightdifferencesurface} The equality
\[ 12g(g-1) \,[k:\qq]\,\h'_\ll(F) = (3g-1)\pair{\hat{\omega},\hat{\omega}} - 2 \sum_{v \in M(k)} \varphi(X_v) \log Nv \]
holds in $\rr$.
\end{thm}
\begin{proof} Let $\Mhat$ denote the adelic line bundle $ \Mhat = 2\hat{\Delta} + \hat{\omega}_1 + \hat{\omega}_2$ on $X^2$. Using that $\deg f_{(1,-1)} \deg_\ll F =  2g(g-1)$ we obtain from Theorem \ref{main} that
\[ \label{aux3} 48g(g-1)\,[k:\qq] \,\h'_\ll(F) = \Mhat^3 \, . 
\] 
We calculate, using Proposition \ref{triplediagonal} and the rules stated above,
\[ \begin{split} \Mhat^3 = & \,\, 8  \langle \hat{\Delta}, \hat{\Delta}, \hat{\Delta} \rangle + 12 \langle \hat{\Delta},\hat{\Delta},\hat{\omega}_1+\hat{\omega}_2 \rangle \\ & + 6 \langle \hat{\Delta}, \hat{\omega}_1+\hat{\omega}_2, \hat{\omega}_1+\hat{\omega}_2 \rangle + \langle \hat{\omega}_1+\hat{\omega}_2, \hat{\omega}_1+\hat{\omega}_2, \hat{\omega}_1+\hat{\omega}_2 \rangle \\
 = & \,\,8  \langle \hat{\Delta}, \hat{\Delta}, \hat{\Delta} \rangle - 24 \pair{\hat{\omega},\hat{\omega}} + 24 \pair{\hat{\omega},\hat{\omega}} + 6(2g-2) \pair{\hat{\omega},\hat{\omega}} \\ = & \,\,8 \pair{\hat{\omega},\hat{\omega}} - 8 \sum_{v \in M(k)} \varphi(X_v) \log Nv+ 12(g-1) \pair{\hat{\omega},\hat{\omega}} 
 \\ = & \,\, (12g-4)\pair{\hat{\omega},\hat{\omega}} - 8 \sum_{v \in M(k)} \varphi(X_v) \log Nv \, . \end{split} \]
The required formula follows. 
\end{proof}
\begin{thm} \label{square} The equality
\[ \begin{split} 12g&(g-1)\, [k:\qq]\,\h'_\ll(Z_{2,\alpha}) \\ & = \frac{3g^2-8g-1}{g-1}\pair{\hat{\omega},\hat{\omega}} + 2 \sum_{v \in M(k)} \varphi(X_v)\log Nv + 48(g-1)(g-2) \,[k:\qq]\,\h'_\ll(x_\alpha) \end{split} \]
holds.
\end{thm}
\begin{proof} Let $\Mhat$ be the adelic line  bundle $ \Mhat = \hat{\omega}_1 + \hat{\omega}_2-2\hat{\Delta} + 2\hat{\beta}_1 + 2\hat{\beta}_2 $
on $X^2$. Using that $\deg f_{2,\alpha} \deg_\ll Z_{2,\alpha}  = 2g(g-1)$ we obtain from Theorem \ref{main} (with $D=2\alpha$) that
\[ 48g(g-1)\,[k:\qq] \,\h'_\ll(Z_{2,\alpha}) = \Mhat^3 \, . \]
It thus suffices to show that
\begin{equation} \label{aux1} \begin{split} \Mhat^3 = &\,\, \frac{12g^2-32g-4}{g-1}\pair{\hat{\omega},\hat{\omega}} \\ & + 8 \sum_{v \in M(k)} \varphi(X_v)\log Nv + 192(g-1)(g-2)\,[k:\qq]\,\h'_\ll(x_\alpha) \, . \end{split} \end{equation}
We recall from equation (\ref{htpoint}) that
\[2 \, [k:\qq] \,\h'_\ll(x_\alpha)=\frac{1}{2g-2}\pair{\hat{\beta},\hat{\omega}} - \frac{1}{(2g-2)^2}\pair{\hat{\omega},\hat{\omega}} \, . \]
Inserting this in (\ref{aux1}) we are reduced to proving that
\begin{equation} \label{aux2} \Mhat^3  = (12g-44)\pair{\hat{\omega},\hat{\omega}} + 8 \sum_{v \in M(k)} \varphi(X_v)\log Nv +48(g-2)\pair{\hat{\beta},\hat{\omega}} \, . \end{equation}
Write
\[ \Mhat_0 = \Mhat + 2 \hat{\Delta}=\hat{\omega}_1 + \hat{\omega}_2 + 2\hat{\beta}_1 + 2\hat{\beta}_2 \, . \]
Then we have
\begin{equation} \label{Mhatexpanded} \Mhat^3 = -8\langle \hat{\Delta}, \hat{\Delta}, \hat{\Delta} \rangle + 12 \langle \hat{\Delta},\hat{\Delta},\Mhat_0 \rangle - 6 \langle \hat{\Delta}, \Mhat_0, \Mhat_0 \rangle + \Mhat_0^3 \, .  
\end{equation}
We compute each of the terms in this expression. From Proposition \ref{triplediagonal} we recall that 
\[ \langle \hat{\Delta}, \hat{\Delta}, \hat{\Delta} \rangle = \pair{\hat{\omega},\hat{\omega}} -  \sum_{v \in M(k)} \varphi(X_v) \log Nv \, . \]
Next we have
\[ \langle \hat{\Delta},\hat{\Delta},\Mhat_0 \rangle = -2\pair{\hat{\omega},\hat{\omega}}-4\pair{\hat{\beta},\hat{\omega}} \]
and, using that $\pair{\hat{\beta},\hat{\beta}}=0$,
\[ \langle \hat{\Delta}, \Mhat_0, \Mhat_0 \rangle = 4\pair{\hat{\omega},\hat{\omega}}+16\pair{\hat{\beta},\hat{\omega}} \, . \]
Finally we have
\[ \begin{split} \Mhat_0^3  = & \,\, 3\left( \langle \hat{\omega}_1,\hat{\omega}_1,\hat{\omega}_2 \rangle + 2 \langle \hat{\omega}_1,\hat{\omega}_1, \hat{\beta}_2 \rangle + \langle \hat{\omega}_2, \hat{\omega}_2, \hat{\omega}_1 \rangle + 2 \langle \hat{\omega}_2,\hat{\omega}_2,\hat{\beta}_1 \rangle \right) \\
 & + 6 \left(  2 \langle \hat{\beta}_1,\hat{\omega}_1,\hat{\omega}_2 \rangle + 2 \langle \hat{\beta}_2, \hat{\omega}_1,\hat{\omega}_2 \rangle + 4 \langle \hat{\beta}_1, \hat{\beta}_2, \hat{\omega}_1 \rangle + 4 \langle \hat{\beta}_1, \hat{\beta}_2, \hat{\omega}_2 \rangle \right) \\
 = & \,\, 3 \left( (2g-2) \pair{\hat{\omega},\hat{\omega}} + 4 \pair{\hat{\omega},\hat{\omega}}  + (2g-2) \pair{\hat{\omega},\hat{\omega}} + 4 \pair{\hat{\omega},\hat{\omega}} \right) \\
 & + 6\left( 2(2g-2) \pair{\hat{\beta},\hat{\omega}} + 2(2g-2)\pair{\hat{\beta},\hat{\omega}} + 8 \pair{\hat{\beta},\hat{\omega}} + 8 \pair{\hat{\beta},\hat{\omega}}\right) \\
  = & \,\,12(g+1)\pair{\hat{\omega},\hat{\omega}}  + 48 (g+1) \pair{\hat{\beta},\hat{\omega}} \, .
\end{split} \]
Summing all contributions with weights as in (\ref{Mhatexpanded}) we find (\ref{aux2}) as required.
\end{proof}

We finish this section by showing that we can write the height of the Gross-Schoen cycle essentially as the difference of the N\'eron-Tate heights of $F$ and $Z_{2,\alpha}$.
\begin{prop} Let $\alpha \in \mathrm{Div}^1 X$. Let $\pair{\Delta_\alpha,\Delta_\alpha}$ denote the height of the Gross-Schoen cycle on $X^3$ associated to $\alpha$. Then the equality
\[  \pair{\Delta_\alpha,\Delta_\alpha} = 3g(g-1)\,[k:\qq]\,(\h'_\ll(F)-\h'_\ll(Z_{2,\alpha}) )+ 12(g-1)^2\,[k:\qq]\,\h'_\ll(x_\alpha)  \]
holds in $\rr$.
\end{prop}
\begin{proof} Theorem \ref{square} and Theorem \ref{heightdifferencesurface} together give
\[ \begin{split} 12g&\,(g-1)\,[k:\qq](\h'_\ll(F)-\h'_\ll(Z_{2,\alpha}) )\\ & =  4\left( \frac{2g+1}{2g-2}\pair{\hat{\omega},\hat{\omega}} - \sum_{v \in M(k)} \varphi(X_v)\log Nv \right) - 48(g-1)(g-2)\, [k:\qq]\,\h'_\ll(x_\alpha) \, . \end{split} \]
Using Zhang's formula (\ref{htgs}) we find
\[ 12g(g-1)[k:\qq](\h'_\ll(F)-\h'_\ll(Z_{2,\alpha}) )= 4 \pair{\Delta_\alpha,\Delta_\alpha} - 48(g-1)^2 \,[k:\qq]\,\h'_\ll(x_\alpha)  \, . \]
The result follows.
\end{proof}
\begin{cor} Assume that $(2g-2)\alpha$ is a canonical divisor. Then
\[  \pair{\Delta_\alpha,\Delta_\alpha} = 3g(g-1)\,[k:\qq]\,(\h'_\ll(F)-\h'_\ll(Z_{2,\alpha}) )\, .  \]
In particular, the equivalence
\[  \pair{\Delta_\alpha,\Delta_\alpha} \geq 0 \Longleftrightarrow \h'_\ll(F) \geq \h'_\ll(Z_{2,\alpha}) \]
holds.
\end{cor}
\begin{remark} When $X$ is hyperelliptic, we have $F=Z_{2,\alpha}$ and we obtain another proof that $\pair{\Delta_\alpha,\Delta_\alpha}= 0$ in this case.
\end{remark}

\section{N\'eron-Tate height of a symmetric theta divisor}
\label{proofautforjac}

Let $X$ be a smooth projective geometrically connected curve of genus $g \geq 1$ with semistable reduction over a number field $k$. Let $(J,\lambda)$ be the jacobian of $X$, and let $\Theta$ be a symmetric effective divisor on $J$ such that $\ll=\oo_J(\Theta)$ is an ample line bundle defining the principal polarization $\lambda$.

Our first aim in this section is to prove the following result. 
\begin{thm} \label{catalyzer} The identity
\[ 2g \,[k:\qq] \, \h'_\ll(\Theta) = \frac{1}{12} \pair{\hat{\omega},\hat{\omega}} + \frac{1}{6} \sum_{v \in M(k)} \varphi(X_v) \log Nv \]
holds in $\rr$.
\end{thm}
\begin{proof} Note that we can write $\Theta=Z_{g-1,\alpha}$ for a suitable $\alpha \in \mathrm{Div}^1 X$ such that  $(2g-2)\alpha$ is a canonical divisor on $X$.  
For $i,j=1,\ldots,g-1$ let $p_i \colon X^{g-1} \to X$ denote the projection on the $i$-th factor, and let $p_{ij} \colon X^{g-1} \to X^2$ denote the projection on the $i$-th and $j$-th factor. Write $\hat{\omega}_i = p_i^*\hat{\omega}$, $\hat{\Delta}_{ij}=p_{ij}^*\hat{\Delta}$ and $\hat{\alpha}_i =p_i^*\hat{\alpha}$.  
We note that
\[ \deg f_{g-1,\alpha} =(g-1)! \quad \textrm{and} \quad \deg_\ll Z_{g-1,\alpha}=g! \, .   \]
By our choice of $\alpha \in \mathrm{Div}^1 X$ we can write $(2g-2)\hat{\alpha}=\hat{\omega} + \hat{\mu}$ where $\hat{\mu}$ is a suitable admissible line bundle on $\Spec k$.
Then, applying Theorem \ref{main} with $D=(g-1)\alpha$ we find
\[ \begin{split} 
2^g (g-1)! & g! g\, [k:\qq] \, \h'_\ll(\Theta)   \\ & = \left\langle \left(\sum_{i=1}^{g-1} \hat{\omega}_i -2\sum_{i<j}^{g-1} \hat{\Delta}_{ij} +(2g-2) \sum_{i=1}^{g-1} \hat{\alpha}_i -(g-1)^2\pair{\hat{\alpha},\hat{\alpha}}\right)^g \bigg| X^{g-1} \right\rangle \\
& = \left\langle \left(2\sum_{i=1}^{g-1}  \hat{\omega}_i -2\sum_{i<j}^{g-1} \hat{\Delta}_{ij} +(g-1) \hat{\mu} -(g-1)^2\pair{\hat{\alpha},\hat{\alpha}}\right)^g \bigg| X^{g-1} \right\rangle \\
& = 2^g \left\langle \left(\sum_{i=1}^{g-1} \hat{\omega}_i -\sum_{i<j}^{g-1} \hat{\Delta}_{ij} \right)^g \bigg| X^{g-1} \right\rangle
+ c' \left( \widehat{\deg} \,\hat{\mu} - (g-1)\pair{ \hat{\alpha},\hat{\alpha}} \right)
\end{split} \]
for some rational number $c'$ depending only on $g$. The equation 
$(2g-2)\hat{\alpha}=\hat{\omega} + \hat{\mu}$ leads by squaring both sides to the identity $4(g-1)(\widehat{\deg} \,\hat{\mu} -(g-1)\pair{ \hat{\alpha},\hat{\alpha}} )= -\pair{\hat{\omega},\hat{\omega}}$, and we conclude that
\[  (g-1)! g! g \,[k:\qq] \, \h'_\ll(\Theta) = \left\langle \left( \sum_{i=1}^{g-1} \hat{\omega}_i -\sum_{i<j}^{g-1} \hat{\Delta}_{ij}\right)^g \bigg| X^{g-1} \right\rangle + c \, \pair{\hat{\omega},\hat{\omega}} \]
for some rational number $c$ depending only on $g$. Now, our considerations in Section~\ref{proofmain} leading to the proof of Theorem \ref{mainintro} show that there exist rational numbers $a'$, $b'$ depending only on $g$ such that
\[ \left\langle \left( \sum_{i=1}^{g-1} \hat{\omega}_i -\sum_{i<j}^{g-1} \hat{\Delta}_{ij}\right)^g \bigg| X^{g-1} \right\rangle =  a' \pair{\hat{\omega},\hat{\omega}} + b' \langle \hat{\Delta}, \hat{\Delta}, \hat{\Delta} \rangle \, . \]
As it turns out, in \cite{wi} precisely the right amount of combinatorics is carried out in order to determine the universal constant $b'$. More precisely, combining \cite[Lemma~5.1(b)]{wi} and \cite[Lemma~5.2]{wi} we find that
\[ b' = -\frac{ g!(g-1)!}{12}  \]
(our $b'$ is $(-1)^g$ times the $a_1'$ from \cite[Lemma~5.1(b)]{wi}). As \[\langle \hat{\Delta}, \hat{\Delta}, \hat{\Delta} \rangle = \pair{\hat{\omega},\hat{\omega}} - \sum_{v \in M(k)} \varphi(X_v) \log Nv\] by Proposition \ref{triplediagonal} we find
that there exists a rational number $a$ depending only on $g$ such that 
\[ 2g\, [k:\qq] \, \h'_\ll(\Theta) = a \, \pair{\hat{\omega},\hat{\omega}} + \frac{1}{6} \sum_{v \in M(k)} \varphi(X_v) \log Nv \, . \]
Specializing to a curve of genus $g$ over some number field $k$ such that the associated jacobian has everywhere good reduction (such curves exist by \cite[Th\'eor\`eme 0.7, Exemple 0.9]{mbskolem}), and using that $\pair{\hat{\omega},\hat{\omega}} $ is always positive, we deduce from Theorem~\ref{realcatalyzer} that $a = \frac{1}{12}$.
\end{proof}

Our next goal is to prove Theorem \ref{autforjac}. By combining Corollary \ref{corofwilms}  and Theorem \ref{catalyzer} we find that
\begin{equation} \label{presymmetric} \begin{split}
\h_F(J) = & \,\, 2g \,\h'_\ll(\Theta) + \frac{1}{[k:\qq]} \sum_{v \in M(k)_0} \frac{1}{12}\left(\delta(X_v) + \vareps(X_v) - 2 \,\varphi(X_v)\right) \log Nv \\ & - \kappa_0 g + \frac{2}{[k:\qq]} \sum_{v \in M(k)_\infty} I(A_v,\lambda_v) \, . \end{split} 
\end{equation}
The issue is therefore to control the contribution from the finite places in this formula. This is taken care of by the next two propositions. 
\begin{prop} \label{graph1} Let $\bar{\Gamma}=(\Gamma,K)$ be a polarized metrized graph. Let $\delta$ be the total length and let $\tau$ be the tau-invariant (\ref{tauinv}) of $\Gamma$. Write $\varphi=\varphi(\bar{\Gamma})$ and $\vareps=\vareps(\bar{\Gamma})$ for the $\varphi$- and $\vareps$-invariant of $\bar{\Gamma}$, cf. (\ref{phiinv}) and (\ref{epsinv}). Then the equality 
\[ \frac{1}{12}(\delta+ \vareps - 2\,\varphi) = \frac{1}{8}\delta - \frac{1}{2} \tau  \]
holds.
\end{prop}
\begin{proof} Let $g$ be the genus of $\bar{\Gamma}$. Let $\mu_a$ be the admissible measure (\ref{canandadmmeasure}) on $\Gamma$ determined by $K$ and let $g_\mu$ denote the associated Green's function (\ref{green}). We will write the invariant $-\vareps + 4g \int_\Gamma g_\mu(x,x) \mu_a$ in two different ways, and compare the outcomes. We recall from (\ref{phiinv}) and (\ref{epsinv}) that
\[ \varphi = -\frac{1}{4}\delta + \frac{1}{4} \int_\Gamma g_\mu(x,x)((10g+2)\mu_a - \delta_K) \]
and 
\[ \vareps = \int_\Gamma g_\mu(x,x)((2g-2)\mu_a + \delta_K ) \, . \]
We find
\begin{equation} \label{firstsum} \begin{split} -\vareps + 4g \int_\Gamma g_\mu(x,x) \mu_a & = \int_\Gamma g_\mu(x,x)((2g+2)\mu_a - \delta_K ) \\
& = \frac{1}{3}\delta - \frac{2}{3} \vareps + \frac{4}{3}\varphi \, . \end{split}
\end{equation}
Let $r(x,y)$ denote the effective resistance between points $x, y$ on $\Gamma$. Then we have
\[ r(x,y) = g_\mu(x,x)-2\,g_\mu(x,y) + g_\mu(y,y) \]
by \cite[Equation (3.5.1)]{zhadm} and this yields, since $\int_\Gamma g_\mu(x,y)\mu_a(y)=0$ by (\ref{green}),
\begin{equation} \label{formulac} \int_\Gamma g_\mu(x,x) \,\mu_a = \frac{1}{2} \int_{\Gamma \times \Gamma} r(x,y)\,\mu_a(x)\,\mu_a(y) \, . 
\end{equation}
By \cite[Lemma 4.1]{mosharp} we have for all $x \in \Gamma$ that
\[  \vareps = 2g\, g_\mu(x,K) + r(x,K) \, , \]
and hence upon integrating both sides against $\mu_a$ we find
\begin{equation} \label{formulaeps} \vareps = \int_{\Gamma \times \Gamma} r(x,y) \,\delta_K(y) \,\mu_a(x) \, .  
\end{equation}
Recall from (\ref{canandadmmeasure}) that $2g\,\mu_a = \delta_K + 2\mu_{\mathrm{can}}$ and from (\ref{tauinv}) that $2\,\tau = \int_\Gamma r(x,y)\mu_{\mathrm{can}}(y)$ independently of the choice of $x \in \Gamma$. Combining (\ref{formulac}) and (\ref{formulaeps}) then gives
\begin{equation} \label{secondsum} \begin{split} -\vareps + 4g \int_\Gamma g_\mu(x,x) \,\mu_a & = \int_{\Gamma \times \Gamma} r(x,y) \,\mu_a(x)\,(-\delta_K +2g\,\mu_a)(y) \\
& = 2 \int_{\Gamma \times \Gamma} r(x,y) \,\mu_a(x) \,\mu_{\can}(y) \\
 & = 4\,\tau \, . 
\end{split}  
\end{equation}
Combining (\ref{firstsum}) and (\ref{secondsum}) we find the required identity.
\end{proof} 
\begin{prop} \label{graph2} Let $\Gamma$ be a metrized graph with total length $\delta$ and tau-invariant~$\tau$. Let $\alpha = \frac{1}{8}\delta - \frac{1}{2}\tau$. Then we have $\alpha \geq 0$, and equality holds if and only if $\Gamma$ is a tree.
\end{prop}
\begin{proof} We refer to \cite[Equation (14.3)]{br} or \cite[Equation (10)]{ciop} for the statement that $ \tau \leq \frac{1}{4}\delta$,
with equality if and only if $\Gamma$ is a tree. This statement can be deduced for example from the explicit formula for $\tau$ in \cite[Proposition 2.9]{ciop}.
\end{proof}
Noting that $J$ has good reduction at a finite place $v \in M(k)$ if and only if the dual graph of the reduction of $X$ at $v$ is a tree, we end up with the following result.
\begin{thm} \label{symmetric} Let $X$ be a smooth projective geometrically connected curve of genus $g \geq 1$ with semistable reduction over $k$. Let $(J,\lambda)$ be the jacobian of $X$. For each $v \in M(k)_0$ let 
\begin{equation} \label{defalphav} \alpha(X_v) = \frac{1}{8}\delta(X_v) - \frac{1}{2} \tau(X_v) \, . 
\end{equation}
Then for each $v \in M(k)_0$ we have $\alpha(X_v) \geq 0$, and we have $\alpha(X_v)=0$ if and only if $J$ has good reduction at $v$. Moreover, the identity
\[ \h_F(J) = 2g \,\h'_\ll(\Theta) +\frac{1}{[k:\qq]} \sum_{v \in M(k)_0} \alpha_v \log Nv- \kappa_0 g + \frac{2}{[k:\qq]} \sum_{v \in M(k)_\infty} I(J_v,\lambda_v) \]
holds in $\rr$.
\end{thm}

\vspace{0.5cm}

\noindent Address of the author:\\ \\
Mathematical Institute  \\
Leiden University \\
PO Box 9512  \\
2300 RA Leiden  \\
The Netherlands  \\ 
Email: \verb+rdejong@math.leidenuniv.nl+

\end{document}